\documentclass[12pt]{amsart}
\usepackage{amsfonts}
\usepackage{amsthm}
\usepackage{amsmath}
\usepackage{xcolor}
\usepackage{amssymb}
\usepackage{amscd}
\usepackage{enumerate}
\usepackage{dsfont} \usepackage[latin2]{inputenc}
\usepackage{t1enc}
\usepackage[mathscr]{eucal}
\usepackage{indentfirst}
\usepackage{graphicx}
\usepackage{graphics}
\usepackage{hyperref}
\usepackage{epic}
\numberwithin{equation}{section}
\usepackage[margin=2.9cm, marginparwidth=2.4cm]{geometry}
\usepackage{epstopdf} 

\usepackage{tikz}
\usetikzlibrary{hobby}

\theoremstyle{plain}
\newtheorem{Th}{Theorem}[section]
\newtheorem{Lemma}[Th]{Lemma}

\newtheorem{Prop}[Th]{Proposition}

 \theoremstyle{definition}

\newtheorem{Rem}[Th]{Remark}
\newtheorem{?}[Th]{Problem}

\newcommand{\Pro}{\mathbb{P}}
\newcommand{\Prob}[1]{\mathbb{P}\{ #1 \}}

\newcommand{\Probx}[2]{\mathbb{P}^{#1}\{ #2 \}}

\newcommand{\ProbBig}[1]{\mathbb{P}\Big\{ #1 \Big\}}
\newcommand{\E}[1]{\mathbb{E} #1 }
\newcommand{\EX}[2]{\mathbb{E}^{#1} #2 }
\newcommand{\EXBig}[2]{\mathbb{E}^{#1}\Big( #2 \Big)}

\newcommand{\EBig}[1]{\mathbb{E} #1 }

\newcommand{\Ind}[1]{\mathds{1}_{\{ #1\}}}

\newcommand{\eps}{\epsilon}
\newcommand{\lam}{\lambda}

\newcommand{\ti}{{\tau_i}}
\newcommand{\tj}{{\tau_j}}

\newcommand{\li}{{\lambda_i}}
\newcommand{\lj}{{\lambda_j}}
\newcommand{\lk}{{\lambda_k}}

\newcommand{\gi}{{\gamma_i}}
\newcommand{\gj}{{\gamma_j}}
\newcommand{\R}{\mathbb{R}}

\newcommand{\fl}{\frac{\lambda_j}{\lambda_i}}

\newcommand{\tfl}{\tfrac{\lambda_j}{\lambda_i}}

\newcommand{\forallj}{\forall j \neq i}

\newcommand{\smallo}[1]{{o\big(#1\big)}}

\usepackage{ulem}
\usepackage{cancel}

\newcommand{\FF}{\mathbf{F}}
\newcommand{\PP}{\mathbb{P}}
\newcommand{\EE}{\mathbb{E}}

\newcommand{\pp}{\mathtt{p}}
\newcommand{\NN}{\mathcal{N}}
\newcommand{\II}{\mathcal{I}}
\newcommand{\Zc}{\mathcal{Z}}

\newcommand{\tZ}{\widetilde{Z}}

\newcommand{\pDD}{\partial_\DD}

\newcommand{\Czero}{\mathcal{C}}

\newcommand{\DD}{\mathbf{D}}
\newcommand{\UU}{\mathbf{U}}
\newcommand{\HH}{\mathcal{H}}
\newcommand{\Int}{\mathop{\mathrm{int}}}
\newcommand{\BB}{\mathcal{B}}
\newcommand{\dist}{\mathrm{dist}}

\newcommand{\Oc}{\mathcal{O}}
\newcommand{\sm}{\setminus}

\newcommand{\indset}{i}
\begin{document}

\title{Atypical Exit Events Near a Repelling Equilibrium}

\author{Yuri Bakhtin}
\author{Hong-Bin Chen}
\address{Courant Institute of Mathematical Sciences\\ New York University \\ 251~Mercer~St, New York, NY 10012 }
\email{bakhtin@cims.nyu.edu, hbchen@cims.nyu.edu }



\begin{abstract}  We consider exit problems for small white noise perturbations of a dynamical system generated by a vector field, and a domain containing 
a critical point  with all positive eigenvalues of linearization. We prove that, in the vanishing noise limit, the probability of exit through a generic set on the boundary is asymptotically polynomial in the noise strength, with exponent depending on the mutual position of the set and the flag of the invariant manifolds associated with the top eigenvalues. Furthermore, we compute the limiting exit distributions conditioned on atypical exit events of polynomially small probability and show that the limits  are Radon--Nikodym equivalent to volume measures on certain manifolds that we construct. This situation is in sharp contrast with the large deviation picture where the limiting conditional distributions are point masses. 
\end{abstract}
	
\maketitle
\tableofcontents

\section{Introduction}

This paper is a part of our program on long-term behavior  of dynamical systems with multiple unstable equilibria organized into heteroclinic networks, under small noisy perturbations. 
The existing work in this direction (see
\cite{Stone-Holmes:MR1050910}, \cite{Stone-Armbruster:doi:10.1063/1.166423}, \cite{Armbruster-Stone-Kirk:MR1964965} for early analysis with elements of heuristics and 
 \cite{Bakhtin2010:MR2731621}, \cite{Bakhtin2011}, \cite{Almada-Bakhtn:MR2802310}, \cite{long_exit_time}, \cite{long_exit_time-1d-part-2}, \cite{exit_time} for rigorous analysis) is a departure from the classical Freidlin--Wentzell (FW) theory of metastability. In FW, rare transitions can be described via large deviations theory and happen at rates exponential in $-\eps^{-2}$ where $\eps$ is the strength of the perturbation:
\begin{align}\label{eq:SDE_X}
    dX^\eps_t = b(X^\eps_t)dt + \eps \sigma(X^\eps_t)dW_t.
\end{align}

In \cite{Kifer1981}, it was shown that the exit from a neighborhood of an unstable critical point of $b$ happens in time of the order of $\log \eps^{-1}$, in the most unstable direction, along the invariant manifold associated to the top eigenvalue of the linearization of the vector field~$b$. In the case where the top eigenvalue is not simple, the limit of the exit  location distribution was studied in \cite{Eizenberg:MR749377}. 
  
In the case of a simple top eigenvalue, the results of~\cite{Kifer1981} were strengthened in \cite{Bakhtin2010:MR2731621}, \cite{Bakhtin2011}, \cite{Almada-Bakhtn:MR2802310}, where  scaling limits for the distributions of  exit locations were obtained  and used to compute asymptotic probabilities of various pathways through the network. In particular, it turned out that there are interesting memory effects and in general the typical limiting behavior at logarithmic timescales is not simply a random walk on the graph of heteroclinic connections.

To study the dynamics over longer times, one has to study the rare events realizing unusual transitions that are improbable over logarithmic time scales, see the discussion of heteroclinic networks in~\cite{long_exit_time-1d-part-2}. It was also understood in \cite{long_exit_time} and~\cite{long_exit_time-1d-part-2} that the leading 
contribution to these rare events is due to abnormally long stays in the neighborhood of the critical point. Asymptotic results
on the decay of probabilities of these events were obtained for repelling equilibria  in these
papers for dimension~$1$ and in~\cite{exit_time}
for higher dimensions. The general results of \cite{exit_time} can be briefly summarized as follows. If all the eigenvalues of the linearization at the critical point are
positive and simple and the leading one 
is $\lambda>0$, then, for
all  $\alpha>1/\lambda$ and initial conditions at distance of the order of $\eps$ to the critical point, 
it was shown that  \[\Prob{\tau>\alpha \log \eps^{-1}}=c\eps^{\beta}(1+o(1)), \quad \eps\to0,\]
where $c$ and $\beta$ were explicitly computed. Note that this is a more precise estimate than $\log \Prob{\tau>\alpha \log \eps^{-1}}=(\beta \log \eps) (1+o(1))$ conjectured in~\cite{mikami1995}. 

 \medskip
 
In the present paper, we extend the study of atypical exit times from~\cite{exit_time} to the study of atypical exit locations in the same setting. We assume that the dynamics near the critical point (which we place at the origin in $\R^d$) admits a smooth conjugacy to the linear dynamics with simple characteristic exponents $\lambda_1>\ldots>\lambda_d>0$ and consider a neighborhood $\DD$ of the origin, with smooth boundary $\partial\DD$. 
For any subset $A$ of $\partial\DD$ possessing a certain regularity property (most relatively open subsets of $\partial\DD$ fall into this category), we prove that 
\begin{equation}
\Prob{X_\tau\in A}=\eps^{\rho(A)}\mu(A)(1+o(1)),\quad \eps\to0,
\label{eq:power-asymptotics}
\end{equation}
where $\rho(A)$ and $\mu(A)$ are constants. The values that the exponent $\rho(A)$ can take belong to a discrete set of values $(\rho_i)_{i=1}^d$: 
\begin{equation}
 \label{eq:rho_i}
\rho_i= \sum_{j<i} \left(\fl -1\right), \quad i=1,2,\ldots,d.
\end{equation}
Here and throughout this paper, the sum over an empty set is understood to be $0$. 

The relevant index $i=\indset(A)$ to be used in \eqref{eq:rho_i}, i.e, such that $\rho(A)=\rho_{\indset(A)}$ in \eqref{eq:power-asymptotics} is defined in the following way. For each $i=1,\ldots,d$, there is a uniquely defined $i$-dimensional manifold $M^i$ invariant under the flow generated by the drift vector field $b$, with tangent space at the origin spanned by eigenvectors associated with eigenvalues $\lambda_1,\ldots,\lambda_i$. These manifolds form a flag, i.e., $M^1\subset M^2\subset\ldots \subset M^d$,
their traces on~$\partial\DD$ defined by $N^i=M^i\cap \partial\DD$ also satisfy $N^1\subset N^2\subset\ldots \subset N^d$ and, additionally, $N^d\cap \partial\DD=\partial\DD$ and thus $\indset(A)=\min\{i\in\{1,\ldots,d\}:\ N^i \cap A\ne\emptyset\}$ is always well-defined, see Figure~\ref{fig:3d_example}. 

Our results mean that exits along manifolds of various dimensions have probabilities of different polynomial decay rates. Since $0=\rho_1<\rho_2<\ldots\rho_d$, these probabilities are of the order of $\eps^{\rho_1}(=\eps^0=1)\gg \eps^{\rho_2}\gg
\ldots\gg\eps^{\rho_d}$.
The differences in the order of magnitude for these probabilities are due to a drastic distortion
caused by exponential expansion with different rates in different eigendirections. One can say that the exit direction of the system is largely determined by its behavior in infinitesimal time which is then amplified by exponential growth with different rates in different directions.
 
In agreement with the results of~\cite{Kifer1981},  exiting in the neighborhood of a two-point set $N^1$ ($M^1$ is a $1$-dimensional manifold, i.e., a curve, associated to the most unstable direction) is a typical event which has asymptotic probability $1$. Exiting away from it happens with probability of the order of $\eps^{\rho_2}$ and, conditioned on this polynomially rare event, the exit distribution concentrates on $N_2$.  In general, exiting away from $N_k$ is a rare event of  probability
of the order of $\eps^{\rho_{k+1}}$ and, conditioned on this rare event, the exit distribution
concentrates on $N_{k+1}$. Moreover,  these conditional distributions 
converge weakly, as $\eps\to 0$, to a limiting measure that is Radon--Nikodym equivalent to the $k$-dimensional volume on $N_{k+1}$, with a density that can be described explicitly. The basic case where the domain~$\DD$ is a cube and the vector field $b$ is linear is at the heart of the analysis.
It turns out that the limiting distributions of exit locations conditioned on exits through various faces of the cube show equidistribution properties that cannot be obtained through large deviation estimates and are surprising if one is used to the FW mindset.

We discuss the simple situation described above and build our intuition in Section~\ref{sec:heuristic}. 
In Section~\ref{section:setting} we give the general setting and our main results in detail. The proofs are given in Sections \ref{sec:derivation-main-result}---\ref{section:extension}.

The techniques that we are using are primarily probabilistic. Most are based on the classical stochastic calculus tools and the key estimate is based on Malliavin calculus. 
In principle, exit problems can be addressed using PDE tools. For exits near unstable critical points, some elements of PDE-based analysis   can be found in~\cite{Kifer1981}, \cite{Day95}, and \cite{CGLM:MR3083930}. So it
would be interesting to find a PDE approach to the problem solved in the present paper but we follow the path of probabilistic analysis using the basic approach similar to \cite{Eizenberg:MR749377}, \cite{Day95}, \cite{Bakhtin2010:MR2731621}, \cite{Bakhtin2011}, \cite{Almada-Bakhtn:MR2802310}, \cite{OnGumbel}, \cite{long_exit_time}, \cite{long_exit_time-1d-part-2}, \cite{exit_time}.

Concluding the introduction, let us briefly discuss two directions that will be natural continuations of the present work.

Although our new results and those on exit times from~\cite{exit_time} are based on 
the same density estimates, we do not develop that connection further in this paper. In particular, the detailed asymptotic analysis of the joint distribution of exit locations and exit times seems possible but harder, and we postpone it to another publication.

A more important question is the asymptotic behavior of exit distributions near hyperbolic critical points (saddles) of the driving vector field. 
Atypical events described in terms of the exit location are responsible for
atypical transitions in heteroclinic networks. Similarly to the situation
in this paper, their probability is expected to decay polynomially in $\eps$ leading to a hierarchy of transitions observable at various polynomial time scales, see the heuristic analysis in~\cite{long_exit_time-1d-part-2}. The approach of the present paper based on Malliavin calculus
density estimates from~\cite{exit_time},
will be an important ingredient in making this analysis rigorous in another forthcoming publication.

{\bf Acknowledgments.} The conditional asymptotic equidistribution first emerged
in discussions with Zsolt Pajor-Gyulai in connection to our project on noisy heteroclinic networks.  YB thanks NSF for the partial support via award DMS-1811444. 

\begin{figure}
    \centering
    \begin{tikzpicture}[use Hobby shortcut,closed=true]
	\draw (0,0) circle (3cm);
	
	\filldraw[dashed,fill=black!5] (0,0) ellipse (3cm and 1cm);
	\filldraw [black] (0,0) circle (1pt);
	\filldraw [black] (-1,0.94) circle (1pt);
	\filldraw [black] (1,-0.94) circle (1pt);
	\draw[dashed] (-1,0.94) -- (1,-0.94);
    
	\node [right] at (0,0.2) {$0$};
	\node [right] at (3.5,0.7) {$i(A_1)=1$};
	\node [right] at (3.5,0) {$i(A_2)=2$};
	\node [right] at (3.5,-0.7) {$i(A_3)=3$};
	\node [right] at (0.7,-1.3) {$A_1$};
	\node [right] at (-1.9,-1.1) {$A_2$};
	\node [right] at (-0.25,2.15) {$A_3$};
	\node [right] at (2.1,2.4) {$\DD\subset \R^3$};

		\draw (1.501, -0.94).. (1.5,-0.9) .. (1.45,-0.6).. (1.1,-0.3) ..(0.5, -0.5).. (0.7,-0.9) .. (0.5, - 1.3).. (1.5,-1.3) ..(1.5, -1.1).. (1.5, -1)..(1.501, -0.94) ;

	\draw (-0.45,-1).. (-0.45,-0.8)..(-1,-0.6).. (-1.5,-0.7)..(-2,-0.5)..(-2.2,-0.45)..(-2.3,-0.6)..(-2.25,-1)..(-2,-1.3)..(-1.5,-1.4)..(-0.9,-1.4)..(-0.55,-1.2)..(-0.45,-1);
	\draw (-1.2,1.8)..(-1,1.5)..(0,1.8)..(0.7,1.6)..(1,1.8)..(1,2.2)..(0,2.55)..(-1,2.3)..(-1.2,1.8);
\end{tikzpicture}
    \caption{The dashed line segment and the shaded surface are the portions of $M^1$ and $M^2$ inside $\DD$, respectively. The sets $A_i$, $i=1,2,3$, lie on $\partial \DD$. Probabilities of exit through these sets have different decay rates: there are constants $c_1,c_2,c_3$ such that $\Prob{X_\tau\in  A_1}\to c_1$,  $\Prob{X_\tau\in A_2}\sim c_2\eps^{\rho_2}$, $\Prob{X_\tau\in A_3}\sim c_3\eps^{\rho_3}$. Moreover, conditioned on  $\{X_\tau\in A_2\}$, the
    distribution of $X_\tau$ has a limit concentrated on  $A_2\cap N_2 = A_2\cap M_2$ and equivalent to the length measure on this curve; conditioned on  $\{X_\tau\in A_3\}$, the
    distribution of $X_\tau$ has a limit equivalent to the area. }
    \label{fig:3d_example}
\end{figure}{}

\section{A heuristic computation for a simple case}
\label{sec:heuristic}

Let us give a heuristic analysis of the simplest situation with exit distribution behavior that is counterintuitive from the point of view of the FW theory.

Suppose the diffusion $X=X^\eps$ in question is two-dimensional:
\begin{align*}
dX^1_t&=\lambda_1 X^1_t dt +\eps dW^1_t,\\
dX^2_t&=\lambda_2 X^2_t dt +\eps dW^2_t,
\end{align*}
where $\lambda_1>\lambda_2>0$ and $X^1_0=X^2_0=0$, and $W^1,W^2$ are independent standard Wiener processes. We define $\tau=\inf\{t\ge 0:\ X_t\in\partial \DD\}$, where $\DD=(-1,1)^2$ is a square and study the distribution of $X_\tau$, the location of exit from $\DD$.

When $\eps$ is small, it takes a long time to exit, and for large times $t$, the Duhamel principle gives
\begin{equation}
\label{eq:duhamel-in-simple-ex}
X^k_t=\eps e^{\lambda_k t}\int_0^{t}e^{-\lambda_k s}dW^k_s \approx \eps e^{\lambda_k t}N_k,
\end{equation}
 where $N_k=\int_0^{\infty}e^{-\lambda_k s}dW^k_s$ is a centered Gaussian random variable with variance $1/(2\lambda_k)$. Denoting $\tau_k=\inf\{t\ge 0: |X^k_t|=1\}$, $k=1,2$, we obtain from
\eqref{eq:duhamel-in-simple-ex} that
\begin{equation}
\label{eq:tau-in-simple-ex}
\tau_k\approx \frac{1}{\lambda_k}\log\frac{1}{\eps}+\frac{1}{\lambda_k}\log\frac{1}{|N_k|}.
\end{equation}
Therefore, for small $\eps$, typically we have $\tau_1<\tau_2$. Moreover,  plugging \eqref{eq:tau-in-simple-ex} with $k=1$ into \eqref{eq:duhamel-in-simple-ex} with $k=2$, we obtain
\[
X_{\tau_1}^2=\eps^{1-\frac{\lambda_2}{\lambda_1}}|N_1|^{-\frac{\lambda_2}{\lambda_1}}N_2\to 0, \quad \eps\to 0,
\]
so the typical random locations of exit $X_\tau$  will concentrate near points $q_\pm=(\pm1, 0)$ where 
the invariant manifold associated with the leading eigenvalue $\lambda_1$ (i.e.,  the first axis) intersects
$\partial\DD$.

Let us now prohibit exits through the sides of $\DD$ that contain $q_\pm$ and study the unlikely event $B$ of exiting $\DD$ through $[-1,1]\times\{-1,1\}$, i.e., we define
$B=\{|X^2_\tau|=1\}$. It turns out that $\Pro(B)=c\eps^{\frac{\lambda_1}{\lambda_2}-1}(1+o(1))$ and the exit distribution conditioned on $B$ is, somewhat surprisingly, asymptotically uniform on $[-1,1]\times\{-1,1\}$. Let us present a heuristic argument for this.

Introducing events 
\[
A_r=\left\{|N_1|<r \eps^{\frac{\lambda_1}{\lambda_2}-1}|N_2|^{\frac{\lambda_1}{\lambda_2}}\right\},\quad r>0, 
\]
we obtain from~\eqref{eq:tau-in-simple-ex} that 
\[
B=\{\tau_2<\tau_1\}\approx A_1,
\]
and, plugging \eqref{eq:tau-in-simple-ex} with $k=2$ for $t$ into \eqref{eq:duhamel-in-simple-ex} with $k=1$, we obtain that
\[
\{|X_{\tau_2}^1|\le r\}\approx A_r,\quad r>0.
\] 
Next,
\[
\Pro(A_r)=\int_{\Sigma_{r,\eps}}g(x_1,x_2) dx_1 dx_2,
\]
where  $\Sigma_{r,\eps}=\{(x_1,x_2):\ |x_1|<r \eps^{\frac{\lambda_1}{\lambda_2}-1}|x_2|^{\frac{\lambda_1}{\lambda_2}} \}$ and $g$ is the joint Gaussian density of~$N_1$ and $N_2$.
As $\eps\to0$, the domain $\Sigma_{r,\eps}$ shrinks to the axis $\{x_1=0\}$, so we
can approximate $g(x_1,x_2)$ by $g(0,x_2)$ and conclude that
\[
\Pro(A_r)=c\eps^{\frac{\lambda_1}{\lambda_2}-1} r(1+o(1)),\quad \eps\to 0,
\]
where
\[
c=2\int_\R g(0,x)|x|^{\frac{\lambda_1}{\lambda_2}}dx.
\]
Therefore,
\[
\Pro(B)=\Pro\{\tau_2<\tau_1\}=c\eps^{\frac{\lambda_1}{\lambda_2}-1}(1+o(1)),\quad \eps \to 0, \]
and
\[
\frac{\Pro\{|X_{\tau_2}^1|< r\}}{\Pro(B)}\to r,\quad \eps\to 0,
\]
which, due to the symmetry of this example, implies that the limiting distribution is uniform.

Another implication of this calculation is that to realize $B$ one needs typical values of $N_2$ and atypically small values of $N_1$. Due to~\eqref{eq:tau-in-simple-ex}, this translates into typical values of $\tau_2$ and atypically large values of $\tau_1$. One can say that the main effect of conditioning on $B$ is conditioning $X^1$ to stay within $[-1,1]$ for abnormally long times, with only a moderate effect on the evolution of $X^2$.

Let us now expand this example and consider a third coordinate evolving independently according to
\[
dX^3_t=\lambda_3 X^3_t dt +\eps dW^3_t,
\]
with $0<\lambda_3<\lambda_2$ and $X^3_0=0$. Now, the unlikely event of interest $B=\{|X^2_\tau|=1\}$  corresponds
to the exit through the union of two faces of the cube $\DD=(-1,1)^3$ given by $[-1,1]\times\{-1,1\}\times[-1,1]$. From the analysis above we know that the exit will happen
at time~$\tau_2$ corresponding to moderate values of $N_2$, i.e, near $\frac{1}{\lambda_k}\log\frac{1}{\eps}$.
Plugging the expression for~$\tau_2$ from \eqref{eq:tau-in-simple-ex} into \eqref{eq:duhamel-in-simple-ex} for $k=3$, we obtain that $X^3_{\tau_2}=\eps^{1-\frac{\lambda_3}{\lambda_2}}|N_2|^{-\frac{\lambda_3}{\lambda_2}}N_3$. Hence, under conditioning on $B$, $X^3_{\tau_2}$
converges to $0$. Combining this with our analysis of the two-dimensional situation above,
we conclude that the exit distribution converges to the 
uniform distribution on $[-1,1]\times \{-1,1\}\times\{0\}$. This union of two one-dimensional segments should be viewed as the intersection of $\partial\DD$ with the two-dimensional invariant manifold associated with $\lambda_1$ and $\lambda_2$, i.e.,  the  $x_1x_2$-plane.

The goal of this paper may be described as to give a rigorous treatment of this example and its 
generalizations to higher dimensional nonlinear situations with space-dependent diffusion matrix and general domains.

\section{Setting and the main result}\label{section:setting}
In $\R^d$, we consider an open simply connected set $\DD$, a bounded vector field $b:\R^d \to \R^d$ 
and the flow $(S_t)_{t\in\R}$ associated with $b$ via the ODE
\begin{align}\label{eq:ODE}
    \begin{split}
        \tfrac{d}{dt}S_tx & = b(S_tx),\\
    S_0x &= x,
    \end{split}
\end{align}
satisfying the following conditions:
\begin{itemize}
    \item the origin $0\in \DD$;
    \item $b(0)=0$;
\item for all $x\in\DD\sm\{0\}$, the deterministic exit time 
\begin{equation}
\label{eq:deterministic-exit-time}
t(x)=\inf\{t\ge0:\ S_tx\in\partial\DD\}
\end{equation}
 satisfies $t(x)<\infty$. In particular, $x=0$ is the only critical point of $b$ in $\DD$;
    \item $b$ is $C^5$ and $b(x)=\mathbf{a}x+q(x)$ where 
        \begin{itemize}
            \item $q(x)\leq C_q|x|^2$ for some $C_q>0$,
            \item $\mathbf{a}$ is a $d\times d$ diagonal matrix with real entries $\lambda_1>\lambda_2>...>\lambda_d>0$;
        \end{itemize}
    \item $\partial \DD$ is $C^1$;
    \item $b$ is transversal to $\partial\DD$, i.e., $ \langle\mathbf{n}(x), b(x)\rangle> 0$ for all $x \in \partial \DD$, where $\mathbf{n}$ denotes the outer normal of $\partial \DD$.
\end{itemize}

A more general situation where $\mathbf{a}$ is only assumed to have eigenvalues $\lambda_1>\lambda_2>...>\lambda_d>0$ can be reduced to this one by a diagonalizing linear transformation.
Our results also hold for a broad class of domains with piecewise smooth boundaries but we restrict ourselves to domains with smooth boundaries for simplicity. 

\bigskip

By the Hartman--Grobman Theorem (see, e.g., \cite[Theorem 6.3.1]{Dynamics}), there is an open neighborhood $\UU$ of $0$ and a homeomorphism $f: \UU \to f(\UU)$  conjugating the flow~$S$ to the linear flow
$\bar S$ generated by the vector field $x\mapsto \mathbf{a}x$ and given by
$\bar S_tx= xe^{\lam t}= (x^je^{\lj t})_{j=1}^d$, namely, 
\begin{align*}
    \frac{d}{dt}f(S_tx)=\mathbf{a}f(S_tx).
\end{align*}
\begin{itemize}
    \item in addition, we assume that $f$ is a $C^5$ diffeomorphism.
\end{itemize}

\begin{Rem}
Due to \cite{Sternberg-1957:MR96853},  for this  $C^5$ conjugacy condition to hold in our setting, it suffices to require (i)~a smoothness condition:
$b$ is $C^k$ for some $k\ge 5 \vee( \lambda_1 /  \lambda_n)$, and (ii)~a no-resonanse condition:
\[\lambda_k\ne m_1\lambda_1+\ldots+m_d \lambda_d\]
for all $k=1,\ldots,d$ and all  nonnegative integer coefficients $m_1,\ldots, m_d$ satisfying $m_1+\ldots+m_d\ge 2$.
\end{Rem}

The vector field $x\mapsto \mathbf{a}x$ is the pushforward of $b$ under $f$, and since $\mathbf{a}$ is diagonal, $f$ can be chosen to satisfy $f(0)=0$ and $Df(0)= I$, the identity matrix.

\bigskip

We are interested in random perturbations of \eqref{eq:ODE} given by~\eqref{eq:SDE_X},
where
\begin{itemize}
    \item $\eps \in (0,1)$ is the noise amplitude parameter;
    \item $(W_t,\mathcal{F}_t)$ is a standard $n$-dimensional Wiener process with $n\geq d$;
    \item $\sigma=(\sigma^i_j)_{i=1,\ldots,d;\ j=1,\ldots,n} $   
    is a map from $\R^d$ into the space of $d\times n$ matrices satisfying
        \begin{itemize}
            \item $\sigma$ is $C^3$ (and , by adjustments outside $\DD$, we may assume that $\sigma$ has bounded derivatives in $\R^d$),
            \item $\sigma(0):\R^n \to \R^d$ is surjective.
        \end{itemize}
\end{itemize}

We will study the solutions of~\eqref{eq:SDE_X} with 
initial data $X^\eps_0 = \eps \xi_\eps\in\mathcal{F}_0$, where 
\begin{itemize}
    \item $\xi_\eps$ converges to some $\xi_0\in \mathcal{F}_0$ in distribution as $\eps \to 0$;
    \item there are constants $C, c>0$ independent of $\eps$ such that
    \begin{align}\label{eq:xi_exponential_tail}
        \Prob{|\xi_\eps|> x}\leq Ce^{-|x|^c} \text{ for all }x\geq 0, \eps \in [0,1).
    \end{align}
\end{itemize}

To simplify notations, we often suppress the dependence on $\eps$. In particular, we write~$X_t$ instead of $X^\eps_t$. We introduce the first time for $X_t$ to exit $\DD$ as 
\begin{align}\label{eq:def_tau}
    \tau=\tau_\eps=\inf\{t>0:X_t\not\in \DD\}.
\end{align}

\medskip

Our main results concern the asymptotic properties of the distribution of $X_\tau$, the location of exit of $X_\tau=X^\eps_{\tau_\eps}$ from $\DD$. To state them, we need to introduce more definitions and notations.

\begin{itemize}
    \item $\BB_L = [-L,L]^d$, \quad $L>0$;
    \item $\FF^i_{L\pm}=\BB_L\cap \{x\in \R^d:x^i=\pm L \}$ is a face of $\BB_L$, and $\FF^i_L=\FF^i_{L+}\cup\FF^i_{L-}$;
    \item for $A\subset B\subset \R^d$, $\partial_{B}A$ and $\Int_{B}A$ denote the boundary and the interior of~$A$ relative to $B$; 
    \item $\partial=\partial_{\R^d}$, $\partial_L=\partial_{\partial \BB_L}$, $\Int_L=\Int_{\partial \BB_L}$;
    \item $\HH^s$ denotes the $s$-dimensional Hausdorff measure.
\end{itemize}

\smallskip

For $k=1,\ldots,d$,  we define sets $\Lambda^k = \bigoplus_{i=1}^k \R e_i$,  where $(e_i)_{i=1}^d$ is the standard basis for~$\R^d$. The sets $\Lambda^k$ are invariant manifolds for the linear flow $(\bar S_t)$ associated with top~$k$ exponents $\lambda_1,\ldots,\lambda_k$. Therefore, the sets
\begin{align}\label{eq:def_M^k}
M^k=\{x\in\R^d: S_{t}x\in f^{-1}(\Lambda^k)\ \text{for some}\ t\in\R\},\quad k=1,\ldots,d,
\end{align}
are the $k$-dimensional invariant manifolds associated with top $k$ exponents for the flow~$(S_t)$. Let us define the traces of these manifolds on the boundary by $N^k=M^k\cap \partial\DD$ and note that due to our transversality assumptions, $N^k$ is a $(k-1)$-dimensional $C^1$-manifold. In particular $N^1$ consists of two points, $N^2$ is a closed curve in $\partial \DD$, and $N^d$ coincides with~$\partial\DD$.

For any set $A\subset\partial \DD$ we define the index of $A$ to be
\begin{equation*}
\indset(A)=\min \big\{ k \in\{ 1,\ldots d\}: \overline A\cap N^k \neq \emptyset \big\},
    \end{equation*}
see Figure~\ref{fig:3d_example}.
This notion is going to be useful because we will show that due to the presence of different exponential growth rates in different directions,  the probabilities for the system to exit~$\DD$ near $N^k$ have different orders of magnitude for different values of $k$. Thus,  the index of~$A$ picks the manifold with the  dominating contribution.  However, this notion becomes truly meaningful and helps computing the asymptotics of exit probabilities only for sets with an additional regularity property which is compatible with the notion of weak convergence of probability measures, holds true for most relatively open subsets of~$\partial\DD$,  and which we proceed to define.

Assuming $d\ge 2$, we say that a set $A\subset\partial \DD$ is $N$-regular if it is Borel and satisfies
\begin{equation}
\label{eq:regularity-def}
\HH^{\indset(A)-1}\{\partial_{\partial \DD} A\cap N^{\indset(A)}\} = 0.
\end{equation}
In the case of $d=1$,  all subsets of  $\partial\DD$ are considered to be $N$-regular.

We still need a few more elements of our construction.

There is  a Euclidean ball $O$ centered  at $0$, satisfying $\overline{f^{-1}(O)}\subset \UU$, and   such that the vector field $x\mapsto\mathbf{a}x$ is transversal to $\partial O$. 
Let us fix~$O$ and define
    \begin{align}\label{eq:def_L(O)}
        L(O)=\sup\{L>0: \BB_L\subset O\}.
    \end{align}
For every $L\in(0,L(O))$, we can define $\psi_L:f^{-1}(\partial\BB_L) \to\partial \DD$  as the Poincare 
map along the flow $(S_t)$:
\begin{equation}
\label{eq:psi_L}
\psi_L(x)=S_{t(x)}x,\quad x\in f^{-1}(\partial\BB_L),
\end{equation}
where $t(\cdot)$ was introduced in~\eqref{eq:deterministic-exit-time}.
We can now define 
\begin{equation}
\zeta_L=f\circ \psi_L^{-1}:\partial\DD\to\partial\BB_L.
\label{eq:def_of-zeta}
\end{equation}

For $x,y\in\R^d$, we define 
\begin{gather*}
\tilde{\chi}^{i}(x^i,y)= \frac{1}{\sqrt{(2\pi)^d\det\Czero}}|x^i+y^i|^{\sum_{j<i}\fl} \int_{ \R^{d-i}}e^{-\frac{1}{2}x^\intercal\mathcal{C}^{-1}_0 x}\Big|_{(x^1,\ldots,x^{i-1})=-(y^1,\ldots,y^{i-1})}dx^{i+1}\ldots dx^d,\\     
           \chi^i_+(y)=\int_{[-y^i,\infty)}\tilde{\chi}^i(x^i,y)dx^i \quad \text{and}\quad \chi^i_-(y)=\int_{(-\infty, -y^i]}\tilde{\chi}^i(x^i,y)dx^i, \end{gather*}
where 
\begin{align}\label{eq:def_C_0_Czero}
    \Czero^{jk}= \sum_{l=1}^n\frac{\sigma^j_l(0)\sigma^k_l(0)}{\lj + \lk}.
\end{align}

For $L<L(O)$ and $i=1,2,\dots,d$, we define the following measure on $\partial \DD$:
\begin{align}\label{eq:def_mu_L}
    \mu^i_L(A)=L^{-\sum_{j<i}\frac{\lj}{\lambda_{i}}}\sum_{\bullet \in \{+,-\}} \EE\chi^i_\bullet(\xi_0)\cdot \HH^{i-1}(\zeta_L(A) \cap \FF^i_{L\bullet}\cap \Lambda^{i} ) , \quad A\subset\partial \DD.
\end{align}
In this definition, the set $\FF^i_{L\bullet}\cap \Lambda^{i}$ is the union of two $(i-1)$-dimensional rectangles:
\[
\FF^i_{L\bullet}\cap \Lambda^{i}=[-L,L]^{i-1}\times\{-L,L\}\times\{0\}^{d-i},
\]
and $\HH^{i-1}(\,\cdot\, \cap \FF^i_{L\bullet}\cap \Lambda^{i})$ is simply the $(i-1)$-dimensional Euclidean volume (Lebesgue measure),
so the measure $\mu^i_L$ is Radon--Nikodym equivalent to the volume measure on of~$N^i\cap \zeta^{-1}(\FF^i_{L})$.

Recalling the definition of $\rho_i$ in~\eqref{eq:rho_i}, we can now state our main result.

\begin{Th} \label{thm:main}
 If $A$ is an $N$-regular set with index $i$, then there is $L_A\in(0,L(O))$  such that
 for all $L\in (0,L_A)$
\begin{align}\label{eq:in_the_main_result}
 \lim_{\eps\to 0}\eps^{-\rho_{i}}\Prob{X_\tau\in A} =\mu^{i}_L(A).
\end{align}
\end{Th}
\begin{Rem}\label{rem:LA-ordered}
 The proof of this theorem also implies that the family of numbers $(L_A)$  indexed by $N$-regular sets $A$ can be chosen to satisfy  $L_{A'}\geq L_A$ for  $A'\subset A$.
\end{Rem}
\begin{Rem} Note that the scaling exponent~$\rho_i$  and the limiting constant $\mu^{i}_L(A)$ in~\eqref{eq:in_the_main_result} are defined explicitly. Thus,  \eqref{eq:in_the_main_result} provides a very precise approximation.
Although the right-hand side of (\ref{eq:in_the_main_result}) seemingly
involves $L$, in fact, it does not depend on $L\in(0,L_A)$.  It is easy to see that
 $L^{-\sum_{j<\indset}\fl}$ in the definition of~$\mu^i_L(A)$ is the correct scaling factor compensating for distortions in directions $1,\ldots, i-1$ introduced by the linear flow that is a part of the definition of $\zeta_L(A)$.
We also note that $N$-regular sets $A$ such that $\mu^i_L(A)>0$ (so Theorem~\ref{thm:main} provides the truly leading
term in the asymptotics) form a large class that includes, for example,  $\zeta_L$-preimages of small open balls with centers in $\FF^i_{L\bullet}\cap \Lambda^{i}$.

\end{Rem}

\medskip
According to Theorem~\ref{thm:main}, the decay rate of probability of exit is the same for all $N$-regular sets of the same index $i$. This, along with the fact that $N$-regular sets are specifically defined to be continuity sets for $\HH^{i-1}$, allows us to state a corollary on the limiting behavior of conditional exit distributions.

If $\Prob{X_\tau \in A}\neq 0$, let $\nu_A^\eps$ be the exit distribution of $X$ conditioned on exiting from $A$:
\begin{align*}
   \nu_A^\eps(\cdot)= \frac{\Prob{X_\tau \in \,\cdot\,\cap A}}{\Prob{X_\tau \in A}}.
\end{align*} 

We denote the weak convergence of finite positive Radon measures by ``$\rightharpoonup$''. 

\begin{Th}\label{Thm:2} Let $A$ be an $N$-regular set of index $i$ and suppose that
 $\mu^{i}_L( A)> 0$. Then for all $L< L_A$ the following weak convergence holds as $\eps \to 0$:
\begin{align*}
    \nu_A^\eps \rightharpoonup \frac{\mu^{i}_L(\,\cdot\, \cap A)}{\mu^{i}_L( A)}.
\end{align*}
\end{Th}

\begin{Rem}
The definition of $\mu^i_L$ in \eqref{eq:def_mu_L} together with the bi-Lipschitzness of $\zeta_L$ and~\eqref{eq:zeta_L_map} implies that $\mu^i_L(\,\cdot\, \cap A)$ is equivalent to the $(i-1)$-dimensional Hausdorff measure restricted to $\overline{ A}\cap N^{i-1}$.
\end{Rem}

In the special case where
$b(x)\equiv \mathbf{a}x$, $\DD=\Int \BB_L$, $A=\FF^i_{L\pm}$ for some $L>0$ and $i=1,\ldots,d$, the limiting measure in Theorem~\ref{Thm:2} is the uniform distribution on $\FF^i_{L\pm} \cap \Lambda^i$. Thus Theorems~\ref{thm:main} and~\ref{Thm:2} are natural generalizations of the simple $2$- and $3$-dimensional 
equidistribution examples discussed in Section~\ref{sec:heuristic}.

  It is important to stress that Theorem~\ref{Thm:2}, where the limiting conditional distribution is equivalent to the volume measure on the manifold $N^i$, paints a picture drastically different from the typical large deviations picture where the limiting conditional distributions are often point masses  concentrated at the minimizers of the large deviation rate function.

The unconditioned exit distribution was also shown to converge to an explicitly computed limit equivalent to the volume on a manifold of smaller dimension in~\cite{Eizenberg:MR749377}. In that paper, the eigenvalues of the linearization are not required to be simple but the assumptions on nonlinearity are fairly restrictive. At the core of the results of \cite{Eizenberg:MR749377} and ours, is the fact that the transition probability over a small time interval is approximately Gaussian and this distribution is carried to the boundary almost deterministically by the flow, different directions being stretched with different rates. However, our results are more delicate since we have to zoom into the transition distribution studying its regularity at small scales with Malliavin calculus tools.

The plan of the proof is the following.

We are going to decompose the dynamics into two stages: (i) the evolution in the transformed coordinates until the exit from a small cube $\BB_L$ (or, equivalently, from $f^{-1}(\BB_L)$ in the original coordinates) and (ii) the evolution between exiting from $f^{-1}(\BB_L)$ and exiting from $\DD$. In the second stage, the process essentially follows the deterministic flow trajectory $(S_t)$ and the associated 
Poincare map $\psi_L$, with error controlled
by a FW large deviation estimate, so it is stage (i) that is central to the analysis. During stage~(i), the evolution is well approximated
by a Gaussian process due to approximate linearity of the drift, so to obtain the desired asymptotics we combine direct computations for this Gaussian process with estimates on the error of the Gaussian approximation based on Malliavin calculus bounds previously obtained in~\cite{exit_time}.

\section{Proof of the main result}\label{sec:derivation-main-result}
Theorem~\ref{thm:main} will follow from two results that we give first.
The first result helps to reduce the problem to considering only sets $A$
with $\zeta_L(A)$ being a subset of the union of two faces of $\BB_L$ associated with coordinate $\indset(A)$, and the second one computes the asymptotic probability of exit through such a set.

\begin{Prop} \label{Prop:geometry_of_pull_back}
  Let $A\subset\partial \DD$  satisfy $\indset(A)=i$. 
  The number~$L'_A$ defined by
  \begin{align*}L'_A=\sup\{L\in(0,L(O)):\ \overline{\zeta_L(A)}\cap \FF^j_L=\emptyset, \forall j<i; \BB_L\subset O\}
  \end{align*}
  is positive, (i.e., the set under the supremum is nonempty)  and for all $L<L'_A$ we have
 \begin{enumerate}
     \item $\overline{\zeta_L(A)}\cap \FF^j_L=\emptyset$ for all $j<i$; \label{item:1_of_prop}
     \item if $i<d$ and $A$ is $N$-regular, then there are $N$-regular $A_0, A_1\subset \partial \DD$ such that \label{item:2_of_prop}
     \begin{enumerate}
         \item $A_0 \subset A \subset A_0\cup A_1$;\label{item:a_of_prop}
         \item $\indset(A_0)=i$; \quad $\zeta_L(A)\cap \Lambda^i = \zeta_L(A_0)\cap \Lambda^i $; \quad$\zeta_L(A_0)\subset  \Int_{\partial\BB_L}{\FF^i_L}$;  \label{item:b_of_prop}
         \item $\indset(A_1)= i+1$.\label{item:c_of_prop} 
     \end{enumerate}
 \end{enumerate}
\end{Prop}

\begin{Prop} \label{Prop:measure-theoretical_prop_of_pull-back}
There is $L_0>0$ such that the following holds. Let $A\subset\partial \DD$ be an arbitrary $N$-regular set with $\indset(A)=i$. If $L<L_0$ satisfies  $\overline{\zeta_L(A)}\subset  \Int_{\partial\BB_L}{\FF^i_L}$, then 
\begin{align*}
    \lim_{\eps\to 0}\eps^{-\rho_i}\Prob{X_\tau \in A} =  \mu^i_L(A) .
\end{align*}
\end{Prop}

Proposition \ref{Prop:geometry_of_pull_back} and Proposition \ref{Prop:measure-theoretical_prop_of_pull-back} will be proved in Sections \ref{subsection:geometry_pullback} and \ref{subsection:error_of_pullback}, respectively.

\begin{proof} [Proof of Theorem~\ref{thm:main}]

Let $L_A=L'_A\wedge L_0$, where $L_0$ and $L_A$ are defined in the propositions above. We immediately see that if $A' \subset A$, then $L_{A'}\geq L_A$, so Remark~\ref{rem:LA-ordered} is automatically justified. 

 The idea of the proof is to use Proposition~\ref{Prop:geometry_of_pull_back} in order to
 approximate $A$ by a union of regular sets of various indices
 such that Proposition~\ref{Prop:measure-theoretical_prop_of_pull-back} can be applied to each of them. More formally, we will use induction on $\indset(A)$, starting with the case $\indset(A) = d$. 
 
 Note that, by (\ref{item:1_of_prop}) of Proposition \ref{Prop:geometry_of_pull_back}, $\indset(A) = d$ implies $\zeta_L(A)  \subset  \Int_{\partial \BB_L}(\FF^d_L)$ for all $L< L_A$. Therefore, we can apply Proposition \ref{Prop:measure-theoretical_prop_of_pull-back} to obtain 
$$\lim_{\eps \to 0} \eps^{-\rho_d}\Prob{X_\tau\in A}  =  \mu^d_L(A), \quad\text{ for all }L< L_A, $$ 
which completes the proof of the induction  basis.

For the induction step, let us assume that the desired result holds for all $A$ with $\indset(A) = k$ for $i+1 \leq k \leq d$. Let us show it is also true for $A$ with $\indset(A) = i$.  

Let us  fix $L\in(0,L_A)$ arbitrarily. Since $i<d$ now, we can define $A_0$ and $A_1$ according to part (\ref{item:2_of_prop}) of Proposition \ref{Prop:geometry_of_pull_back}. Since $A_0\subset A$, Remark \ref{rem:LA-ordered} implies  $L<L_A\leq L_{A_0}$. Then, using part~(\ref{item:b_of_prop}) of Proposition \ref{Prop:geometry_of_pull_back},  Proposition \ref{Prop:measure-theoretical_prop_of_pull-back}, and the definition of $\mu^i_L$ in \eqref{eq:def_mu_L}, we obtain
\begin{align}
    \lim_{\eps \to 0} \eps^{-\rho_i}\Prob{X_\tau\in A_0} =\mu^i_L(A_0)=\mu^i_L(A). 
\label{eq:asymptotics-for-A_0}
\end{align}
By (\ref{item:c_of_prop}) of Proposition \ref{Prop:geometry_of_pull_back},  $i(A_1)=i+1$, so by the induction hypothesis, for each $L'\leq L_{A_1}$,
\begin{align*}
    \lim_{\eps \to 0} \eps^{-\rho_{i+1}}\Prob{X_\tau\in A_1}  =  \mu^{i+1}_{L'}(A_1),
\end{align*}
which implies that $\Prob{X_\tau\in A_1} = \mathcal{O}(\eps^{\rho_{i+1}}) = \smallo{\eps^{\rho_i}}$.
Due to (\ref{item:a_of_prop}),
\begin{align*}
    |\Prob{X_\tau\in A} - \Prob{X_\tau\in A_0}| \leq \Prob{X_\tau\in A_1} = \smallo{\eps^{\rho_i}}.
\end{align*}
Combining this with~\eqref{eq:asymptotics-for-A_0}, we complete the induction step and the entire proof.
\end{proof}

To prove Theorem \ref{Thm:2}, we need the following basic result.

\begin{Lemma} \label{lemma:Hyperbolic_Pull_Back}
Let $A\subset \partial \DD$ be arbitrary. Then the following holds: 
\begin{enumerate}
\item \label{item:equiv-index} $\indset(A) = \min\big\{k\in \{1,\ldots,d\}:\overline{\zeta_L(A)}\cap \Lambda^k\neq \emptyset \big\}$ for each $L<L(O)$;
\item \label{item:equiv-Lambda-regular} if $A$ is Borel, then the following statements are equivalent:
\begin{enumerate}
    \item $A$ is $N$-regular,
    \item $\HH^{\indset(A)-1}\{\partial_L(\zeta_L(A))\cap \Lambda^{\indset(A)}\}=0$ for some $L<L(O)$,
    \item $\HH^{\indset(A)-1}\{\partial_L(\zeta_L(A))\cap \Lambda^{\indset(A)}\}=0$ for all $L<L(O)$.
\end{enumerate}
\end{enumerate}
\end{Lemma}

\begin{proof} 
First,  $\zeta_L$ is a  bi-Lipschitz homeomorphism, since it is a composition of a diffeomorphism $f$ and the Poincar\'e map $\psi_L^{-1}$ constructed in~\eqref{eq:psi_L} from smooth flows transversal to locally smooth sections. Secondly, due to \eqref{eq:def_M^k}, the definition $N^k=M^k\cap \partial \DD$, and the invariance of $\Lambda^k$ under the linear flow $\bar S$, one can see that
\begin{align}\label{eq:zeta_L_map}
     \zeta_L(\overline{A}\cap N^k)=\overline{\zeta_L(A)}\cap \Lambda^k \quad\text{and}\quad \zeta_L(\partial_{\partial \DD}A\cap N^k)= \partial_L(\zeta_L(A))\cap \Lambda^k,
\end{align}
which implies  both parts \eqref{item:equiv-index} and \eqref{item:equiv-Lambda-regular} straightforwardly.
\end{proof}

\begin{proof}[Proof of Theorem \ref{Thm:2}]
We need to prove that 
\begin{equation}
\nu_A^\eps (B)\to  \frac{\mu^{i}_L(B \cap A)}{\mu^{i}_L( A)},\quad \eps\to 0,
\label{eq:conv-on-continuity-sets}
\end{equation}
for every continuity set $B$ of the measure $\mu_L^i(\,\cdot\, \cap  A)$ or, equivalently, by the definition~\eqref{eq:def_mu_L}, of  $\HH^{i-1}\big(\zeta_{L}(\,\cdot\, \cap A)\cap\FF^{i}_{L\pm}\cap \Lambda^{i} \big)$ which is equal to $\HH^{i-1}\big(\zeta_{L}(\,\cdot\, \cap A)\cap \Lambda^{i} \big)$ due to $L<L_A$. 
Using the inclusion $\partial_\DD(B\cap A)\subset (\pDD B \cap \overline A)\cup (\overline B\cap \pDD A)$,  
the  $N$-regularity of~$A$ (see~\eqref{eq:regularity-def}) and \eqref{item:equiv-Lambda-regular} of Lemma \ref{lemma:Hyperbolic_Pull_Back}, we conclude that the continuity property of~$B$ implies that
of $B\cap A$. Combining this with the fact that $\nu_A^\eps(B\cap A^c)=0$ for all~$\eps$, we obtain that it is sufficient to 
check~\eqref{eq:conv-on-continuity-sets} for Borel subsets $B$ of $A$ with continuity property. For such a set $B$,
either  $\indset(B)=\indset(A)$, or $\indset(B)>\indset(A)$. In the first case, 
writing
\[
\pDD B = \pDD(B\cap A)\subset (\pDD B\cap \overline A) \cup (\overline B \cap \pDD A),
\]
using the continuity of $B$, part~\eqref{item:equiv-Lambda-regular} of Lemma~\ref{lemma:Hyperbolic_Pull_Back}, and the $N$-regularity of $A$, we conclude that~$B$ is also $N$-regular, so \eqref{eq:conv-on-continuity-sets} follows from Theorem~\ref{thm:main}.   In the second case,  part~\eqref{item:equiv-index} of Lemma \ref{lemma:Hyperbolic_Pull_Back} 
implies 
$\overline{\zeta_L(B)}\cap \Lambda^{i} =\emptyset$. Therefore, $r=\dist(\zeta_L(B),\Lambda^i)>0$, where
\begin{equation}
\label{eq:dist-def}
\dist(C,D)=\inf\{|x-y|:\ x\in C,\ y\in D\}\wedge 1,\quad C,D\subset \R^d.
\end{equation}
Since $\HH^{i-1}\big(\zeta_{L}(B\cap A)\cap\FF^{i}_{L\pm}\cap \Lambda^{i} \big)=0$, 
it suffices to prove $\nu^\eps_A(B)\to0$
 to ensure~\eqref{eq:conv-on-continuity-sets}.  
Let us define  $B_L^i(r)=\{x\in\BB_L: \dist(\{x\},\Lambda^i)\ge r\}$. 
Since $\zeta_L(B)\subset B_L^i(r)$, and $\zeta^{-1}_L(B_L^i(r))$ is an $N$-regular
set of index $i+1$ due to parts \eqref{item:equiv-index} and \eqref{item:equiv-Lambda-regular} of Lemma \ref{lemma:Hyperbolic_Pull_Back}, 
we can apply Theorem~\ref{thm:main} to $\zeta^{-1}_L(B_L^i(r))$ and 
conclude that $\Prob{X_\tau \in B}=\smallo{\eps^{\rho_{\indset(A)}}}=\smallo{\Prob{X_\tau \in  A}}$,
so~\eqref{eq:conv-on-continuity-sets} holds in this case as well. The proof is completed.
\end{proof}

\section{Exit from a box}
Recall the definitions of $\BB_L=[-L,L]^d$ and $\FF^i_{L\pm}$ in Section \ref{section:setting}. Let
\begin{align*}\begin{split}
    \FF^i_{L\pm,\delta} & = \{x \in \FF^i_{L\pm}: |x^j| \leq L - \delta \text{ for $j \neq i$}  \}, \\
    \FF^i_{L,\delta} &=\FF^i_{L+,\delta}\cup \FF^i_{L-,\delta}.
\end{split}
\end{align*}
Set 
\begin{align}\label{eq:def_tau_L}
    \tau_L = \inf \{t>0: X_t \notin f^{-1}(\BB_L) \}, 
\end{align}
where $f$ is the linearizing conjugacy. The main result of this section  gives the exit probability asymptotics for sets whose images under $f$ are rectangles:

\begin{Prop} \label{Prop:Box_Case}
There exists $L_0 >0$ such that for all positive $L\leq L_0$, we have 
\begin{align*}\lim_{\eps\to 0}\eps^{-\rho_i}\Prob{X_{\tau_L}\in f^{-1}(A )}  = & L^{-\sum_{j<i}\fl}\sum_{\bullet \in\{ +,-\}}\EE\chi^i_\bullet(\xi_0)  \HH^{i-1}(A \cap \FF^i_{L\bullet}\cap \Lambda^i )
\end{align*}
for all $i\in\{1,2,\dots,d\}$, and for all sets $A$ with the following properties:
\begin{itemize}
    \item $A$ is a product of intervals which can be open, closed, or half-open.
    \item $\overline{A}=[a^1,b^1]\times \ldots \times [a^{i-1},b^{i-1}]\times \{\pm L\}\times[a^{i+1},b^{i+1}]\times\ldots\times[a^d,b^d]$, where $[a^j,b^j] \subset (-L,L)$ for all $j\neq i$ and $a^j,b^j\ne 0$ for $j>i$.
\end{itemize}
\end{Prop}

\subsection{Derivation of Proposition~\ref{Prop:Box_Case} from auxiliary results} From now on, we use standard summation convention over matching upper and lower indices.
Let $Y_t = f(X_t)$, where $f$ is the linearizing conjugacy. Using It\^o's formula, we obtain 
\begin{align} \label{eq:Y_SDE_before_Duhamel}
\begin{split}
        dY^i_t & = \mathbf{a}^i_j Y^j_tdt + \eps \partial_k f^i(f^{-1}(Y_t))\sigma^k_j(f^{-1}(Y_t))dW^j_t\\
        &+ \frac{\eps^2}{2}\sum_{j,k=1}^d\partial_{j,k}^2f^i(f^{-1}(Y_t))\langle \sigma^j(f^{-1}(Y_t)),\sigma^k(f^{-1}(Y_t)) \rangle dt \\
    & = \lambda^i Y^i_t dt + \eps F^i_j(Y_t) dW^j_t + \eps^2G^i(Y_t)dt,
\end{split}
\end{align}
where 
\begin{itemize}
    \item $\lambda^i = \lambda_i$ to avoid summation in $i$;
    \item $F$ and $G$ are $C^3$ (since $f$ is  $C^5$ and $\sigma$ is $C^3$);
    \item since $f(0)=0$ and $Df(0)=I$, we have $F(0)= \sigma(0)$. \end{itemize}

\medskip
Since $F(0)=\sigma(0)$ is $d\times n$ with full rank and $F$ is continuous, we can find $L_0>0$ small so that there is $c_0>0$ such that $\min_{|u|=1,u\in\R^d}|u^\intercal F(x)|^2\geq c_0$ for all $x\in [-L_0,L_0]^d$. We shrink $L_0$ further, if necessary, to ensure $L_0\leq L(O)$ as in \eqref{eq:def_L(O)}. Since we will only care about exiting from a subset of $[-L_0,L_0]^d$, we modify $F, G$ outside $[-L_0,L_0]^d$ so that
\begin{align}\label{eq:modified_F}
\begin{split}
    \min_{|u|=1,u\in\R^d}|u^\intercal F(x)|^2\geq c_0, \text{ for all }x\in \R^d; \\
    F,G \text{ and their derivatives are bounded}.
    \end{split}
\end{align}
 
With this $L_0$ chosen, we will consider the following for the rest of this section, applying Duhamel's principle to (\ref{eq:Y_SDE_before_Duhamel}) and setting $Y_0 = \eps y$, 
\begin{align}\label{eq:Y_after_Duhamel}
\begin{split}
    Y^j_t&=\eps e^{\lj t} y^j+\eps e^{\lj t}\Big(\int_0^t e^{-\lj s}F^j_l(Y_s)dW^l_s+\eps\int^t_0e^{-\lj s}G^j(Y_s)ds \Big)\\
    & = \eps e^{\lj t}(y^j+ M^j_t + \eps V^j_t)=\eps e^{\lj t}(y^j+ U^j_t ),     \end{split}
\end{align}
where $F,G$ are modified to ensure \eqref{eq:modified_F}. 
We emphasize that $M_t$, $V_t$ and $U_t$ all depend on $y$ and $\eps$. We define $\PP^{\eps y}= \PP\{\,\cdot\, |Y_0=\eps y\}$.

\smallskip

Let $C_f$ be the Lipschitz constant of $f$ and $c_f=C_f^{-1}$.
Since $f(0)=0$, we have 
$|\eps^{-1}f(\eps x)|\leq C_f|x|$ for all $x$ and $\eps$.
  In view of \eqref{eq:xi_exponential_tail}, we choose $\kappa>0$ large so that for 
\begin{align}\label{eq:def_K(eps)}
    K(\eps)= (\log\eps^{-1})^{\kappa}
\end{align} we have 
\begin{align}\label{eq:xi>K(eps)_is_eps^rho_d}
    \Prob{|\xi_\eps|>c_fK(\eps)}\leq \eps^{\rho_d+\delta} \text{, for all }\eps \in (0,1]\text{, for some }\delta >0.
\end{align}
By our definition of $c_f$, we have
\begin{align*}
    \text{if }|x|\leq c_f K(\eps), \quad\text{ then }|\eps^{-1}f(\eps x)|\leq K(\eps).
\end{align*}
\begin{Rem}\label{remark:eta_and_kappa} Later, when needed, $\kappa$ in \eqref{eq:def_K(eps)} will be adjusted to be even larger. This will not affect our results.
\end{Rem}

According to~\eqref{eq:def_tau_L},  $\tau_L = \inf\{t>0: Y_t\notin \BB_L \}$. To prove Proposition \ref{Prop:Box_Case}, we first obtain asymptotics for $Y_t$ exiting rectangular sets uniformly in $Y_0 = \eps y$ with $|y|\leq K(\eps)$. Recall that $\rho_i$ is given in \eqref{eq:rho_i}.

\begin{Prop}\label{prop:Rectangle_set_on_box} 
Consider $Y_t$ defined by \eqref{eq:Y_after_Duhamel}. If $L< L_0$, $i=1,\ldots,d$, 
and $A$ is a rectangle described in Proposition \ref{Prop:Box_Case}, then
\begin{align}\label{eq:lim_sup_K(eps)}
    \lim_{\eps\to 0}\sup_{|y|\leq  K(\eps)}\bigg|&\eps^{-\rho_i}\Probx{\eps y}{Y_{\tau_L}\in A}- \chi^i_\pm\big(\eps^{-1}f^{-1}(\eps y)\big) c_A \bigg|= 0,
\end{align}
where
\begin{align}\label{eq:def_c_A}
    c_A=L^{-\sum_{j<i}\fl}\prod_{j<i}(b^j-a^j)\prod_{j>i}\Ind{0\in(a^j,b^j)}.
\end{align}
\end{Prop}
Here and throughout the paper, the product over an empty set is understood to equal~$1$.

\begin{proof}[Derivation of Proposition \ref{Prop:Box_Case} from Proposition~\ref{prop:Rectangle_set_on_box}]
Consider \eqref{eq:SDE_X} with $X_0=\eps \xi_\eps$ described in \eqref{eq:xi_exponential_tail} and observe that, by \eqref{eq:xi>K(eps)_is_eps^rho_d} and the above proposition,
\begin{align}
\label{eq:provingProp:Box_Case-1}
    \eps^{-\rho_i}\Prob{X_{\tau_L} \in f^{-1}(A)} &= \EE \eps^{-\rho_i}\Probx{f(\eps\xi_\eps)}{Y_{\tau_L} \in A}\Ind{|\xi_\eps|\leq c_fK(\eps)} + o(1)\\ \notag
    & = c_A\EE \chi^i_\pm(\xi_\eps)\Ind{|\xi_\eps|\leq c_fK(\eps)}+o(1).
\end{align}
There is $C>0$ such that
\begin{align}\label{eq:polynomial_bound_of_chi}
    |\chi^i_\pm(y)|\leq C(1+|y|^{\sum_{j<i}\fl}),\quad x\in\R^d.
\end{align}
Due to the fast decay of the tail of $\xi_\eps$ imposed by \eqref{eq:xi_exponential_tail}, all positive moments of $|\chi^i_\pm(\xi_\eps)|$ are bounded uniformly in $\eps$. Therefore, due to  \eqref{eq:xi>K(eps)_is_eps^rho_d} and H\"older's inequality, we have $\EE\chi^i_\pm(\xi_\eps)\Ind{|\xi_\eps|>c_fK(\eps)} = o(1)$, which implies
\begin{align}
\label{eq:provingProp:Box_Case-2}
    c_A\EE \chi^i_\pm(\xi_\eps)\Ind{|\xi_\eps|\leq c_fK(\eps)}= c_A\EE \chi^i_\pm(\xi_\eps)+o(1).
\end{align}
This, along with the uniform tail bound on $\xi_\eps$ in~\eqref{eq:xi_exponential_tail},
 the polynomial bound on $\chi_{\pm}$ in \eqref{eq:polynomial_bound_of_chi}, 
and continuity of $\chi_\pm$,
implies
\begin{align}
\label{eq:provingProp:Box_Case-3}
    \lim_{\eps\to 0}c_A\EE \chi^i_\pm(\xi_\eps) = c_A\EE\chi^i_\pm(\xi_0).
\end{align}
Combining~\eqref{eq:provingProp:Box_Case-1}, \eqref{eq:provingProp:Box_Case-2}, and \eqref{eq:provingProp:Box_Case-3}, we obtain
\begin{align*}
    \lim_{\eps\to 0} \eps^{-\rho_i}\Prob{X_{\tau_L} \in f^{-1}(A)} = L^{-\sum_{j<i}\fl}\EE\chi^i_\pm(\xi_0)\HH^{i-1}\{A\cap \Lambda^i\},
\end{align*}
completing the proof.
\end{proof}

\subsection{Proof of Proposition \ref{prop:Rectangle_set_on_box}}
Let us fix any $L<L_0$ and $i\in\{1,2,\dots,d\}$. 
First, we remark that it suffices to consider $A$ satisfying
\begin{align}\label{eq:rectangle_main_type}
    0\in(a^i,\ b^i),\quad \forall j >i.
\end{align}
In fact, if $A$ does not satisfy \eqref{eq:rectangle_main_type}, then we can find two rectangles $A'$ and $A''$ such that
\begin{enumerate}[(i)]
    \item $\pi^{\leq i}(A)=\pi^{\leq i}(A')=\pi^{\leq i}(A'')$ where $\pi^{\leq i}$ is the projection onto the first $i$ coordinates;
    \item $A'\subset A''$ and $A\subset A''\setminus A'$;
    \item $A'$ and $A''$ satisfy \eqref{eq:rectangle_main_type}.
\end{enumerate}
By \eqref{eq:def_c_A}, we have $c_A=0$ since $A$ does not satisfy \eqref{eq:rectangle_main_type},  and $c_{A'}=c_{A''}$ due to (i). These together with (ii) imply (we use $\eta_\pm =\chi^i_\pm(\eps^{-1}f^{-1}(\eps y))$):
\begin{align*}
    &\big|\eps^{-\rho_i}\Probx{\eps y}{Y_{\tau_L}\in A}-\eta_\pm c_A\big|=\big|\eps^{-\rho_i}\Probx{\eps y}{Y_{\tau_L}\in A}\big|\\
    &\leq \big|\eps^{-\rho_i}\Probx{\eps y}{Y_{\tau_L}\in A'}-\eta_\pm c_{A'}\big|+\big|\eps^{-\rho_i}\Probx{\eps y}{Y_{\tau_L}\in A''}-\eta_\pm c_{A''}\big|.
\end{align*}
Finally, (iii) allows us to apply \eqref{eq:lim_sup_K(eps)} to $A'$ and $A''$, and thus \eqref{eq:lim_sup_K(eps)} holds for $A$.

To avoid heavy notation, we also assume that $A$ is closed. It can be readily checked that all our arguments are still valid if $A$ is not closed.

Recall $\tau_L$ given in \eqref{eq:def_tau_L}. Since $L$ is fixed, for brevity, we write $\tau=\tau_L$ for the rest of the section.  Here, we only study the case 
where $A\subset\FF^i_{L+}$, which corresponds to
$Y^i_\tau=L$. The case where $A\subset\FF^i_{L-}$ (corresponding to $Y^i_\tau=-L$) can be considered in the same way.

We will need the following two statements.

\begin{Lemma} \label{lemma:rough_est} Assume $0\in (a^j,b^j)$ for all $j>i$. Let 
\begin{align}\label{eq:def_T_0}
    T_0=T_0(\eps)= \frac{1}{\li}\log\frac{L}{\eps(\log\eps^{-1})^{\kappa+1}}.
\end{align}
There are  $\gamma_j$, $j=1,\ldots,d$, satisfying
 \begin{align} \label{eq:gamma_j_condition}
     0\vee \Big(\fl-1\Big)<\gj <\fl  \quad j=1,2,\dots,d,
 \end{align}
such that
$$ \Prob{y+U_{T_0}\in B_-} - \smallo{\eps^{\rho_i}} \leq \Prob{Y^j_\tau\in [a^j,b^j],\forallj; Y^i_\tau = L} \leq \Prob{y+U_{T_0}\in B_+} + \smallo{\eps^{\rho_i}}$$ holds uniformly in $|y| \leq K(\eps)$, where
\begin{align} \label{eq:defOfB_pm}
\begin{split}
B_\pm &=\cup_{x^i \in I_{\pm}}\big(B_{\pm,<i}^{(x^i)} \times \{x^i\} \times B^{(x^i)}_{\pm,>i}\big) \\&= \cup_{x^i \in I_{\pm}}\big((J_{\pm,1}^{(x^i)}\times...\times J_{\pm,i-1}^{(x^i)})\times \{x^i\}\times (J_{\pm,i+1}^{(x^i)}\times ... \times J_{\pm,d}^{(x^i)})\big)\\
\end{split}
\end{align} 
with $I_\pm = \big(\mp\eps^{\gi}, ( \log\eps^{-1})^{\kappa + 1}\pm\eps^{\gi}\big]$, 
and for $j\neq i$
\begin{align*}
    J_{\pm,j}^{(x^i)} &= \Big[a^jL^{-\fl}\eps^{\fl - 1}(|x^i|\pm \eps^{\gi})^\fl \mp \eps^{\gj},b^jL^{-\fl}\eps^{\fl - 1}(|x^i|\pm\eps^{\gi})^\fl\pm\eps^{\gj}\Big].
\end{align*}
\end{Lemma}
Note that due to~\eqref{eq:gamma_j_condition}, for small $\eps>0$, the terms $\eps^{\gamma_j}$ are small compared to the leading order terms in the definition
of $J_{\pm,j}^{(x^i)}$.

\begin{Lemma}\label{lemma:iteration}
Let $T_0$ be given in \eqref{eq:def_T_0} and $\Zc$ be a centered Gaussian vector with covariance matrix $\Czero$ given in \eqref{eq:def_C_0_Czero}. Then for each $\upsilon \in(0,1)$, there is  $\delta>0$ such that
 \begin{align*}
    \sup_{|y|\leq \eps^{\upsilon-1}} \big|\Prob{y+U_{T_0}\in B_\pm}-\Prob{y+\Zc\in B_\pm}\big| = \smallo{\eps^{\rho_i+\delta}}.
 \end{align*}
\end{Lemma}

\medskip

Let us define
\begin{align}
    h_\eps(y)&= \eps^{-\rho_i}\Prob{y+\Zc\in B_\pm },\label{eq:def_h_eps_initial} \\
    h_0(y) & =L^{-\sum_{j<i}\fl}\chi^i_\pm(y)\prod_{j<i}(b^j-a^j) \label{eq:def_h_0_initial}.
\end{align}
Note the dependence on $\pm$, which is suppressed to avoid heavy notation.
Proposition \ref{prop:Rectangle_set_on_box} follows from Lemmas \ref{lemma:rough_est}, \ref{lemma:iteration} and the following estimate:
\begin{align}\label{eq:est_h_eps(y) - h_0(eps^-1_f^-1(eps y ))}
    \sup_{|y|\leq K(\eps)}\big|h_\eps(y) - h_0(\eps^{-1}f^{-1}(\eps y ))\big| = o(1).
\end{align}
Our plan is to derive \eqref{eq:est_h_eps(y) - h_0(eps^-1_f^-1(eps y ))} in the remainder of this subsection and then prove Lemmas~\ref{lemma:rough_est} and \ref{lemma:iteration} in Subsection~\ref{sec:proofs-of-auxiliary}. 

\subsubsection{Proof of \eqref{eq:est_h_eps(y) - h_0(eps^-1_f^-1(eps y ))}}

We split \eqref{eq:est_h_eps(y) - h_0(eps^-1_f^-1(eps y ))} into estimating $|h_\eps(y) - h_0(y)|$ and $|h_0(y) - h_0(\eps^{-1}f^{-1}(\eps y ))|$ separately. The techniques involved are elementary but the proof is tedious. We proceed in steps. 

\smallskip

Step 1. We express $h_\eps$ and $h_0$ explicitly in the form of Gaussian integrals over some sets. For each $x\in \R^d$, let
\begin{gather}
    x^{<i} = (x^1,...,x^{i-1}), \quad x^{\geq i} = (x^{i},...,x^d), \quad x^{>i} = (x^{i+1},...,x^d), \quad \hat{x}=(x^{<i},x^{>i});\nonumber\\
    \hat{B}^{(x^i)}_\pm = B_{\pm,<i}^{(x^i)}\times B_{\pm,>i}^{(x^i)}\label{eq:def_hatB}
\end{gather}
where $B_{\pm,<i}^{(x^i)}$ and $B_{\pm,>i}^{(x^i)}$ are given in \eqref{eq:defOfB_pm}.
When $y$ is fixed as in $Y_0=\eps y$, we write $\tilde{x}=(-y^{<i},x^{\geq i})\in \R^d$ for each $x\in\R^d$. 

Now, let us introduce
\begin{align}
    g_\eps(x^i,y) &=\frac{\eps^{-\rho_i}}{\sqrt{(2\pi)^d\det \Czero}}\int_{\hat{B}^{(x^i+y^i)}_\pm -\hat{y}}e^{-\frac{1}{2}x^\intercal\Czero^{-1} x}d\hat{x},\label{eq:def_g_eps}\\
    g_0(x^i,y)&=\frac{\prod_{j<i}(b^j-a^j)L^{-\fl}|x^i+y^i|^\fl}{\sqrt{(2\pi)^d\det\Czero}}\int_{ \R^{d-i}} e^{-\frac{1}{2}\tilde{x}^\intercal\mathcal{C}^{-1}_0 \tilde{x}}dx^{>i}.\label{eq:def_g_0}
\end{align}
Recall the definition of $I_\pm$ below \eqref{eq:defOfB_pm}. Additionally, we set
\begin{align}\label{def_interval_II}
    \II_+(y^i)=[-y^i,\infty) \quad \text{and}\quad \II_-(y^i)=(-\infty,-y^i].
\end{align}
Using the definitions \eqref{eq:def_h_eps_initial}--\eqref{eq:def_h_0_initial}, we can see
\begin{align*}
    h_\eps(y)&= \int_{I_{\pm}-y^i} g_\eps(x^i,y)dx^i,\\
    h_0(y) &= \int_{\II_\pm(y^i)} g_0(x^i,y)dx^i.
\end{align*}

\smallskip

Step 2. We record some useful estimates, which will be proved later. For simplicity of notation, we write
\begin{align}\label{eq:def_z_eps}
    z_\eps(y)=\eps^{-1}f^{-1}(\eps y).
\end{align}
For convenience, we set $z_0(y)=y$. The following holds for all $\eps\in[0,1]$ and $x^i\in\R$,
\begin{align}
    & \sup_{|y|\leq K(\eps)}|g_\eps(x^i,y)|\leq C\big(1+K(\eps)\big)^p e^{-c|x^i|^2}, \label{eq:boundedness_of_g_eps}\\
    &\sup_{|y|\leq K(\eps)}|g_\eps(x^i,y)-g_0(x^i,y)|\leq C \eps^{q} e^{-c|x^i|^2}, \label{eq:est_g_eps_-_g_0}\\
    &\sup_{|y|\leq K(\eps)}|y-z_\eps(y)|\leq C\eps^q,\label{eq:y^i-z^i_eps_est}\\
    &\sup_{|y|\leq K(\eps)}|g_0(x^i,y)-g_0(x^i,z_\eps(y))| \leq C \eps^{q} e^{-c|x^i|^2}, \label{eq:prep_4_h_0_est}
\end{align}
for some $C,c,p,q>0$.

\medskip

Step 3. We estimate $|h_\eps(y)-h_0(y)|$ for $|y|\leq K(\eps)$. We shall only treat the case where $\pm$ is $+$ and $\mp$ is $-$. The other case is similar. We start by writing
\begin{align*}
    |h_\eps(y) - h_0(y)|\leq &\int_{I_+ - y^i}|g_\eps(x^i,y)-g_0(x^i,y)|dx^i+ \int_{-\eps^{\gi}-y^i}^{-y^i}|g_\eps(x^i,y)|  dx^i \\
    &+ \int_{(\log\eps^{-1})^{\kappa+1}+\eps^\gi-y^i}^\infty|g_0(x^i,y)|dx^i.
\end{align*}
Using \eqref{eq:boundedness_of_g_eps} and \eqref{eq:est_g_eps_-_g_0}, we have, for some $q'>0$,
\begin{align}\label{eq:est_h_eps(y) - h_0(y)}
\begin{split}
    \sup_{|y|\leq K(\eps)}|h_\eps(y) - h_0(y)| \leq & \int_\R C \eps^{q} e^{-c|x^i|^2}dx^i \quad +C\eps^\gi (1+K(\eps))^p\\
    &+\int^\infty_{(\log\eps^{-1})^{\kappa+1}+\eps^\gi-K(\eps)}C (1+K(\eps))^pe^{-c|x^i|^2}dx^i = o(\eps^{q'}).
\end{split}
\end{align}

\smallskip

Step 4. We estimate $|h_0(y)-h_0(z_\eps(y))|$ for $|y|\leq K(\eps)$. Recall the definition of $\II_\pm(y^i)$ in \eqref{def_interval_II} and note that, due to \eqref{eq:y^i-z^i_eps_est},
\begin{align}\label{eq:symmetric_difference_est}
    \big|\II_\pm(y^i)\triangle\II_\pm(z_\eps^i(y))\big|\leq |y^i-z^i_\eps(y)|\leq C\eps^q.
\end{align}
Here $\triangle$ denotes the symmetric difference of sets. By the formula for $h_0(y)$ in Step~1, we first write
\begin{align*}
&|h_0(y)-h_0(z_\eps(y))|\\
&\leq \int_{\II_\pm(z_\eps^i(y))}|g_0(x^i,y)-g_0(x^i,z_\eps(y))|dx^i+\int_{\II_\pm(y^i)\triangle\II_\pm(z_\eps^i(y))}|g_0(x^i,y)|+|g_0(x^i,z_\eps(y))| dx^i.
\end{align*}
We can bound $|g_0(x^i,z_\eps(y))|$ by using \eqref{eq:boundedness_of_g_eps} and \eqref{eq:prep_4_h_0_est}. Apply this, \eqref{eq:prep_4_h_0_est} and \eqref{eq:symmetric_difference_est} to see
\begin{align}\label{eq:est_h_0(y) - h_0(z)}
\sup_{|y|\leq K(\eps)}|h_0(y)-h_0(z_\eps(y))|\leq \int_\R C\eps^{q}e^{-c|x^i|^2}dx^i +  C\eps^q(1+K(\eps))^p=o(\eps^{q''})
\end{align}
for some $q''>0$.

\smallskip

In conclusion, \eqref{eq:est_h_eps(y) - h_0(eps^-1_f^-1(eps y ))} follows from \eqref{eq:est_h_eps(y) - h_0(y)} and \eqref{eq:est_h_0(y) - h_0(z)}. 

\bigskip

It remains to prove estimates listed in Step~2. We prove them in the following order: \eqref{eq:y^i-z^i_eps_est}, \eqref{eq:boundedness_of_g_eps},  \eqref{eq:prep_4_h_0_est}, \eqref{eq:est_g_eps_-_g_0}.

\begin{proof}[Proof of \eqref{eq:y^i-z^i_eps_est}]
Recall the definition of $z_\eps(y)$ in \eqref{eq:def_z_eps} and that the local diffeomorphism~$f$ satisfies $f(0)=0$ and $Df(0)=I$. Since $|y|\leq K(\eps)$, if $\eps$ is small, then $|\eps y|$ is uniformly close to $0$. By expanding $f^{-1}$ at $0$, we can see that there are $C,q>0$ such that
\begin{align*}
    |y-z_\eps(y)|=|y - \eps^{-1}(f^{-1}(\eps y))|\leq C\eps|y|^2\leq C\eps^q,\quad \forall|y|\leq K(\eps).
\end{align*}
This gives \eqref{eq:y^i-z^i_eps_est}.
\end{proof}

\begin{proof}[Proof of \eqref{eq:boundedness_of_g_eps}]
Since $\sigma(0)$ has full rank, by definition of $\Czero$ in \eqref{eq:def_C_0_Czero}, there is  $c>0$ such that 
\begin{align}\label{eq:gaussian_density_bound}
    e^{-\frac{1}{2}x^\intercal\Czero^{-1}x }\leq e^{-c|x|^2},\qquad \forall x\in\R^d.
\end{align}
Hence, there is $C>0$ such that 
\begin{align*}
        |g_0(x^i, y)|&\leq C |x^i+y^i|^{\sum_{j<i}\fl}e^{-c|x^i|^2}.
\end{align*}
Absorb polynomials of $x^i$ into the exponential to see, for some $C,c,p>0$,
\begin{align*}
    \sup_{|y|\leq K(\eps)}|g_0(x^i, y)|\leq C\big(1+K(\eps)\big)^p e^{-c|x^i|^2}.
\end{align*}
From this and \eqref{eq:est_g_eps_-_g_0}, we obtain \eqref{eq:boundedness_of_g_eps}.
\end{proof}

\begin{proof}[Proof of \eqref{eq:prep_4_h_0_est}]

We simplify the expression \eqref{eq:def_g_0} into
\begin{align*}
    g_0(x^i,y)=C|x^i+y^i|^{p_0}\int_{\R^{d-i}}e^{-\frac{1}{2}x^\intercal \Czero^{-1}x}\big|_{x^{<i}=-y^{<i}}dx^{>i},
\end{align*}
for some $C,p_0>0$. Then, we have
\begin{align}\label{eq:prep_6_h_0_est}
\begin{split}
    &|g_0(x^i,y)-g_0(x^i,z_\eps(y))|\\
    &\quad\leq C\Big(|x^i+y^i|^{p_0}-|x^i+z^i_\eps(y)|^{p_0}\Big)\int_{\R^{d-i}}e^{-\frac{1}{2}x^\intercal \Czero^{-1}x}\big|_{x^{<i}=-z_\eps^{<i}(y)}dx^{>i}\\
&\qquad   + C|x^i+y^i|^{p_0}\int_{\R^{d-i}}\bigg|e^{-\frac{1}{2}x^\intercal \Czero^{-1}x}\big|_{x^{<i}=-y^{<i}}-e^{-\frac{1}{2}x^\intercal \Czero^{-1}x}\big|_{x^{<i}=-z_\eps^{<i}(y)}\bigg|dx^{>i}.
\end{split}
\end{align}

Let us estimate the terms on the right of \eqref{eq:prep_6_h_0_est}. Using \eqref{eq:y^i-z^i_eps_est}, we have, for some $C,p,q>0$, 
\begin{align}\label{eq:x^i+y^i_est}
    \Big| |x^i+y^i|^{p_0}-|x^i+z^i_\eps(y)|^{p_0}\Big| \leq C (|x^i|^p+|y|^p)\eps^{q'} \leq C\eps^q(|x^i|^p+1), \quad\forall|y|\leq K(\eps).
\end{align}
By \eqref{eq:gaussian_density_bound}, there are $C,c>0$ such that
\begin{align}\label{eq:g_0_est_prep_2}
    \int_{\R^{d-i}}e^{-\frac{1}{2}x^\intercal \Czero^{-1}x}\big|_{x^{<i}=-y^{<i}}dx^{>i}\leq Ce^{-c|x^i|^2}.
\end{align}

To estimate the integrand of the last integral in \eqref{eq:prep_6_h_0_est}, we need the following observation. Since $\Czero$ is symmetric and positive definite, there are $C,C'>0$ such that, for all $w,z\in\R^d$,
\begin{align}
        &\big| e^{-\frac{1}{2}w^\intercal\Czero^{-1}w}-e^{-\frac{1}{2}z^\intercal\Czero^{-1}z}\big| \leq C (e^{-c|w|^2}\vee e^{-c|z|^2})|w+z||w-z|\nonumber\\
    &\quad \leq C e^{-c|w|^2}(|2w|+|w-z|)|w-z|\Ind{|w|\leq |z|}\nonumber\\
    &\qquad\qquad+C e^{-c|z|^2}(|2z|+|w-z|)|w-z|\Ind{|w|> |z|}\nonumber\\
    &\quad \leq C'(e^{-c'|w|^2}+e^{-c'|z|^2})(|w-z|+|w-z|^2).\label{eq:difference_gaussian_density}
\end{align}
Using this estimate and \eqref{eq:y^i-z^i_eps_est}, we can obtain
\begin{align}
        &\,\Big|e^{-\frac{1}{2}x^\intercal \Czero^{-1}x}\big|_{x^{<i}=-y^{<i}}-e^{-\frac{1}{2}x^\intercal \Czero^{-1}x}\big|_{x^{<i}=-z_\eps^{<i}(y)}\Big|\nonumber\\
    \leq&\, e^{-c|x^{\geq i}|^2}\big(|y^{<i}-z_\eps^{<i}(y)|+|y^{<i}-z_\eps^{<i}(y)|^2\big)\nonumber\\
    \leq&\, C\eps^qe^{-c|x^{\geq i}|^2},\qquad \forall |y|\leq K(\eps),\label{eq:g_0_est_prep_3}
\end{align}
Insert \eqref{eq:x^i+y^i_est}, \eqref{eq:g_0_est_prep_2}, 
and~\eqref{eq:g_0_est_prep_3} to the right hand side of \eqref{eq:prep_6_h_0_est} to see
\begin{align*}
    &|g_0(x^i,y)-g_0(x^i,z_\eps(y))|\\
&\leq C\eps^q(|x^i|^p+1)e^{-c|x^i|^2}+C(|x^i|^{p_0}+|K(\eps)|^{p_0})\eps^qe^{-c|x^i|^2}\\
&\leq C\eps^{q'}e^{-c'|x^i|^2}, \qquad \forall|y|\leq K(\eps),
\end{align*}
for some $c',q'>0$. This completes the proof.

\end{proof}

\begin{proof}[Proof of \eqref{eq:est_g_eps_-_g_0}]
Again, the techniques involved are elementary while the proof is tedious. Recall $g_\eps$ and $g_0$ in \eqref{eq:def_g_eps}--\eqref{eq:def_g_0}, and the notation in \eqref{eq:def_hatB}. To estimate the difference between $g_\eps$ and $g_0$, we introduce
\begin{align*}\begin{split}
              &\mathtt{I}= \frac{\eps^{-\rho_i}}{\sqrt{(2\pi)^d\det \Czero}}\int_{\hat{B}_\pm^{(x^i+y^i )}-\hat{y} }e^{-\frac{1}{2}\tilde{x}^\intercal \Czero^{-1}\tilde{x}}d\hat{x},\\
        &\mathtt{II}=\frac{\prod_{j<i}(b^j-a^j)L^{-\fl}|x^i+y^i|^{\fl}}{\sqrt{(2\pi)^d \det \Czero}}\int_{B_{\pm,>i}^{(x^i+y^i )}-y^{>i}} e^{-\frac{1}{2}\tilde{x}^\intercal \Czero^{-1}\tilde{x}}dx^{> i}.
    \end{split}
\end{align*}
Then, we write
\begin{align}\label{eq:split_difference_between_g_eps-g_0}
    \begin{split}
        &\quad\quad\big|g_\eps(x^i,y)-g_0(x^i,y)\big| \leq \big|g_\eps(x^i,y)-\mathtt{I}\big|+\big|\mathtt{I}-\mathtt{II}\big|+\big|\mathtt{II}-g_0(x^i,y)\big|.
    \end{split}
\end{align}
We proceed in steps. In each step, we estimate one term on the right of the above display. 

\medskip

Step 1. We estimate $|g_\eps(x^i,y)-\mathtt{I}|$ for $|y|\leq K(\eps)$. We start by writing
\begin{align*}
    |g_\eps(x^i,y)-\mathtt{I}| \leq C\eps^{-\rho_i}\int_{\hat{B}_\pm^{(x^i+y^i )}-\hat{y} }| e^{-\frac{1}{2}x^\intercal\Czero^{-1}x}-e^{-\frac{1}{2}\tilde{x}^\intercal\Czero^{-1}\tilde{x}}\big| d\hat{x}
\end{align*}
Let us estimate the integrand. Recall the estimate \eqref{eq:difference_gaussian_density}. Using $\fl-1>0$ for $j<i$, $|y|\leq K(\eps)$, and $e^{-c|x|^2}$ to absorb powers of $|x^i|$, we obtain that, for all
$x,y$ satisfying $\hat{x}\in \hat{B}_\pm^{(x^i+y^i )}-\hat{y}$ and $|y|\leq K(\eps)$, 
\begin{align*}\begin{split}
        &\big| e^{-\frac{1}{2}x^\intercal\Czero^{-1}x}-e^{-\frac{1}{2}\tilde{x}^\intercal\Czero^{-1}\tilde{x}}\big|  \leq Ce^{-c_1|x^{\geq i}|^2}\big(|x^{<i}+y^{<i}|+|x^{<i}+y^{<i}|^2\big)\\
        &\leq Ce^{-c_1|x^{\geq i}|^2}\sum_{j<i}(|v^j|+|v^j|^2)\Big|_{v^j=\big|\eps^{\fl -1}(|x^i+y^i |+ \eps^\gi)^\fl+ \eps^\gj \big|}\leq C \eps^{q_1}e^{-c_2|x^{\geq i}|^2},
    \end{split}
\end{align*}
for some $C, c_1,c_2, q_1>0$.
This, along with the definitions of $\hat{B}^{(x^i+y^i)}_{\pm}$ in \eqref{eq:def_hatB}, $B^{(x^i+y^i)}_{\pm,<i}$ in \eqref{eq:defOfB_pm} and $\rho_i$ in \eqref{eq:rho_i}, implies
\begin{align}\label{eq:est_1st_difference}
    \begin{split}
        &|g_\eps(x^i,y)-\mathtt{I}| \leq C\eps^{-\rho_i}\int_{\hat{B}_\pm^{(x^i+y^i )}-\hat{y} } \eps^{q_1}e^{-c_2|x^{\geq i}|^2}d\hat{x}\\
        &= C\eps^{q_1-\rho_i}e^{-c_2|x^i|^2}\bigg(\int_{B_{\pm,<i}^{(x^i+y^i )}-y^{<i} }dx^{<i} \bigg)\bigg(\int_{B_{\pm,>i}^{(x^i+y^i )}-y^{>i} }e^{-c_2|x^{> i}|^2} dx^{>i}\bigg)\\
        &\leq C\eps^{q_1-\rho_i}\big|B_{\pm,<i}^{(x^i+y^i)}\big|\, e^{-c_2|x^i|^2}\\
        &\leq C\eps^{q_1-\rho_i}\eps^{\sum_{j<i}\fl-1}(|x^i|+1)^{\sum_{j<i}\fl}e^{-c_2|x^i|^2}\leq C\eps^{q_1}e^{-c_3|x^i|^2}.
    \end{split}
\end{align}
Here and henceforth we use $|B|$ to denote the Lebesgue measure of a set $B$.

\medskip
Step 2. We estimate $|\mathtt{I}-\mathtt{II}|$. First note that, by integrating over the first $i-1$ coordinates in~$\mathtt{I}$ and the definition of $B^{(x^i,y^i)}_{\pm,<i}$ in \eqref{eq:defOfB_pm}, we have
\begin{align*}
    \mathtt{I} = \frac{\prod_{j<i}\big((b^j-a^j)L^{-\fl}(|x^i+y^i |\pm   \eps^\gi)^\fl\pm 2\eps^{\gj-(\fl - 1)}\big)}{\prod_{j<i}(b^j-a^j)L^{-\fl}|x^i+y^i|^\fl}\mathtt{II}.
\end{align*}
Also, $|\mathtt{II}|\leq C \prod_{j<i}|x^i+y^i|^\fl e^{-c|x^i|^2}$ for some $C>0$. Hence, using $|y|\leq K(\eps)$ and $e^{-c|x^i|^2}$ to absorb powers of $|x^i|$, we can obtain, for some $C,c, q_2>0$,
\begin{align}\label{eq:est_2nd_difference}
    \begin{split}
        \big|\mathtt{I}-\mathtt{II}\big|\leq C\eps^{ q_2}e^{-c|x^i|^2}.
    \end{split}
\end{align}

Step 3. We estimate $|\mathtt{II}-g_0(x^i, y)|$. Note that
\begin{align*}
    |\mathtt{II}-g_0(x^i,y)|&\leq C\int^\infty_{\R^{d-i}\setminus (B_{\pm,>i}^{(x^i+y^i )}-y^{>i})}|x^i+y^i|^{\sum_{j<i}\fl}e^{-c|x^{\geq i}|^2}dx^{> i}\\
        &\leq e^{-c|x^i|^2}\sum_{j> i}\int_{\R \setminus (J^{(x^i+y^i )}_{\pm,j}-y^j )} |x^i+y^i|^{\sum_{j<i}\fl}e^{-c|x^j|^2}dx^j.
\end{align*}
We split the integrals after the last inequality into
\begin{align*}
    &\int_{\R \setminus (J^{(x^i+y^i )}_{\pm,j}-y^j )} |x^i+y^i|^{\sum_{j<i}\fl}e^{-c|x^j|^2}dx^j \\
        &= \int_{b^jL^{-\fl}\eps^{\fl-1}(|x^i+y^i |\pm\eps^\gi)^\fl \pm \eps^\gj-y^j }^\infty |x^i+y^i|^{\sum_{j<i}\fl}e^{-c|x^j|^2}dx^j\\
        &+\int^{a^jL^{-\fl}\eps^{\fl-1}(|x^i+y^i |\pm\eps^\gi)^\fl \mp \eps^\gj-y^j }_{-\infty}
    |x^i+y^i|^{\sum_{j<i}\fl}e^{-c|x^j|^2}dx^j.
\end{align*}
Choosing $q'>0$ sufficiently small, we consider two cases. If $|x^i+y^i|\leq \eps^{q'}$, then for some $q_3>0$, the above display is bounded by $C\eps^{q_3}$. For the case where $|x^i+y^i|> \eps^{q'}$, let us set $c_j= (|a^j|\wedge|b^j|)L^{-\fl}$ and recall that $\fl -1 <0$ for $j>i$.  For some $q_4,q_5,p>0$, the above display can be bounded by
\begin{align*}
    &2\int^\infty_{c_j\eps^{\fl -1 }(\eps^{q'}-\eps^\gi)^{\fl} - \eps^\gj - K(\eps)}|x^i+y^i|^{\sum_{j<i}\fl}e^{-c|x^j|^2}dx^j\\
    &\leq C\eps^{q_4}|x^i+y^i|^{\sum_{j<i}\fl}\leq C\eps^{q_5}(|x^i|^p+1)\qquad \forall |y|\leq K(\eps).
\end{align*}
In deriving the above inequality, we have used $\fl -1 <0$ and chosen $q'<\gi$. Combining the above, we have
\begin{align}\label{eq:est_5th_difference_case_1}
        |\mathtt{II}-g_0(x^i,y)|&\leq C\big(\eps^{q_3}+\eps^{q_5}(1+|x^i|^p)\big)e^{-c|x^i|^2}\leq C\eps^{q_6}e^{-c'|x^i|^2}.
\end{align}

\bigskip

To conclude, we insert \eqref{eq:est_1st_difference}, \eqref{eq:est_2nd_difference}, and  \eqref{eq:est_5th_difference_case_1} into \eqref{eq:split_difference_between_g_eps-g_0}. As a consequence, we obtain that, for some constants $q', C, c'>0$, the following holds for $\eps$ sufficiently small,
\begin{align*}
    \begin{split}
        \sup_{|y|\leq K(\eps)}&|g_\eps(x^i,y)-g_0(x^i,y)|\leq C \eps^{q'} e^{-c'|x^i|^2},
    \end{split}
\end{align*}
as desired.

\end{proof}

\subsection{Proofs of Lemmas \ref{lemma:rough_est} and \ref{lemma:iteration}}
\label{sec:proofs-of-auxiliary}
\begin{proof}[Proof of Lemma \ref{lemma:rough_est}]
Let $\tau_j = \inf\{t> 0: |Y^j_t| = L \}$.  We recall \eqref{eq:def_tau_L} and the notation $\tau=\tau_L$. Hence, we have $\tau = \min_{j=1,2,...,d}\{\tau_j \}$. First, we show the following.

\begin{Lemma} \label{lemma:drop_tau_i}
If $[a^j,b^j]\subset(-L,L)$ for all $j\neq i$, then, with $\rho_i$ defined in \eqref{eq:rho_i},
\begin{align*}
    \sup_{|y|\leq K(\eps)}\big|\Prob{Y^j_\tau \in [a^j, b^j], \forallj;Y^i_\tau =L  } - \Prob{Y^j_\ti \in [a^j, b^j], \forallj;Y^i_\ti =L } \big | = \smallo{\eps^{\rho_i}}.
\end{align*}
\end{Lemma}

\begin{proof}
Since
\[
\Prob{Y^j_\tau \in [a^j, b^j], \forallj;Y^i_\tau =L  }  = \Prob{Y^j_\ti \in [a^j, b^j], \forallj;Y^i_\ti =L ; \tau =\ti },
\]
it remains to estimate the right-hand side of 
\begin{multline*}
    \Prob{Y^j_\ti \in [a^j, b^j], \forallj;Y^i_\ti =L} - \Prob{Y^j_\ti \in [a^j, b^j], \forallj;Y^i_\ti = L;\tau = \ti  }\\
 = \Prob{Y^j_\ti \in [a^j, b^j], \forallj;Y^i_\ti = L;\ti > \tau  }.
\end{multline*}
Using the strong Markov property and setting $c_j=|a^j|\vee|b^j|$, we can bound it by
\begin{align*}
& \sum_{j\neq i} \Prob{|Y^j_\ti|\leq c_j, \ti > \tj} \\
    &\leq \sum_{j\neq i}  \EBig{\sum_{l = \pm L}\Ind{Y^j_\tj = l}\Probx{Y_\tj}{ |Y^j_{\tau_i}|\leq c_j}}
    \leq \sum_{j\neq i,\ l = \pm L}  \EBig{\Ind{Y^j_\tj = l}\Probx{Y_\tj}{ e^{\lj \ti}|l+\eps U^j_{\ti}|\leq c_j}}\\ 
    &\leq \sum_{j\neq i} \EBig{\Ind{Y^j_\tj = l}\Probx{Y_\tj}{ L-\eps |U^j_\ti|\leq c_j}}
    \leq\sum_{j\neq i} \eps^p(L-c_j)^{-p}\E{|U^j_\ti|^p} 
\end{align*}
for any $p>0$. Let $p >\rho_i$. 
By \eqref{eq:modified_F}, there is $C>0$ such that, for all $j$, almost surely,
\begin{align}\label{eq:<M>_and_|V|_bound}
\sup_{t\in[0,\infty)}\langle M^j \rangle_t\le C\quad \text{ and} \quad\sup_{t\in[0,\infty)}|V^j_t| \leq C.
\end{align}
This, along with  BDG inequality, implies that $\E{|U^j_\ti|^p}$ is bounded uniformly in $\eps$, and completes the proof.
\end{proof}

We will approximate $U_{\tau_i}$ by $U_{T_0}$, where $T_0$ is given \eqref{eq:def_T_0}. By \eqref{eq:Y_after_Duhamel},
\begin{align} \label{eq:TimeGaussianRelation}
L=|Y^i_\ti|=\eps e^{\li \ti}|y^i+U^i_\ti|, \quad\text{ or } \quad\ti = \frac{1}{\li}\log\frac{L}{\eps |y^i+U^i_\ti|}.
\end{align}
Now \eqref{eq:Y_after_Duhamel} and (\ref{eq:TimeGaussianRelation}) give
\begin{align*}
    Y^j_\ti = L^\fl \eps^{1-\fl}(y^j+U^j_\ti )|y^i+U^i_\ti|^{-\fl},
\end{align*}
which implies
\begin{multline}\label{eq:MainProbability}
\Prob{Y^j_\ti\in [a^j,b^j],\forallj;\ Y^i_\ti = L} = \Prob{Y^j_\ti\in [a^j,b^j],\forallj;\ Y^i_\ti > 0} \\ = \ProbBig{y^j+U^j_\ti \in L^{-\fl}\eps^{\fl - 1}|y^i+U^i_\ti|^\fl[a^j,b^j],\forallj;\ y^i+U^i_\ti>0}.
\end{multline}

Then, we compare $\tau_i$ with $T_0$ by showing that, for an appropriate choice of $\kappa$,
\begin{align}\label{eq:ti>t_0_whp}
    \Prob{\ti<T_0} = \Prob{|y^i+U^i_\ti| >(\log\eps^{-1})^{\kappa+1}} = \smallo{\eps^{\rho_i}}.
\end{align}
By \eqref{eq:<M>_and_|V|_bound} and the exponential martingale inequality (see Problem 12.10 in \cite{Bass}), the following holds uniformly in $|y|\leq K(\eps)$ and $\eps$ sufficiently small,
\begin{align*}
\Prob{\ti <T_0}& = \Prob{|y^i+U^i_\ti| >(\log\eps^{-1})^{\kappa+1}; \ti<T_0} \leq \Prob{|y^i+U^i_{ \ti\wedge T_0}|>(\log\eps^{-1})^{\kappa+1}}\\
& \leq \Prob{|M^i_{\ti \wedge T_0}|>( \log \eps^{-1})^{\kappa+1}-( \log \eps^{-1})^{\kappa}-C\eps}\\
&\leq \Prob{|M^j_{\ti \wedge T_0}|>\tfrac{1}{2}( \log \eps^{-1})^\kappa}\\
& \leq 2\exp\big(-(8C)^{-1}( \log \eps^{-1})^{2\kappa}\big).
\end{align*}
Therefore, it suffices to choose $\kappa$ large enough
(see Remark \ref{remark:eta_and_kappa}) to guarantee \eqref{eq:ti>t_0_whp}. 
So, with high probability, $\ti \geq T_0$. Let us choose  $\delta$ to satisfy
\begin{align*}
    0< \delta < 2\tfrac{\lambda_d}{\lambda_i}= 2\min_{j}\{ \tfl \} < 2.
\end{align*} 
Using the boundedness of  $F$, and $G$, we can write for some $C_\delta>0$:
\begin{multline}
\label{eq:quad-var--M}
\langle M^j \rangle_{\ti\vee T_0} - \langle M^j \rangle_{T_0}  = \int_{T_0}^{\ti \vee T_0}e^{-2\lj T_0}|F^j (Y_{s\wedge\tau})|^2 ds\\
\leq Ce^{-2\lj T_0} 
\leq C\eps^{2\fl}( \log \eps^{-1})^{2\fl(\kappa +1)} \leq C_\delta \eps^{2\fl-\delta}
\end{multline}
and
\begin{align}
\label{eq:bound-on-V}
| V^j_{\ti\vee T_0} -  V^j_{T_0} |
\leq Ce^{-\lj T_0}\leq C_\delta \eps^{\fl-\frac{1}{2}\delta}.
\end{align}
Then we can choose $\gamma_j>0$, $j\in\{1,2,\dots,d\}$ to satisfy, as anticipated in \eqref{eq:gamma_j_condition},
 \begin{align}
 \label{eq:cond-on-gamma}
     0\vee \Big(\fl-1\Big)<\gj <\fl - \frac{1}{2}\delta.
 \end{align} 
By the exponential martingale inequality, estimates \eqref{eq:quad-var--M}, \eqref{eq:bound-on-V}, and the second inequality in~\eqref{eq:cond-on-gamma},
 we have
\begin{align}\label{eq:U^j_ti_U^j_t_0_difference_est}
\begin{split}
    &\Prob{|U^j_{\ti \vee T_0} - U^j_{T_0}|>\eps^\gj} \leq \Prob{|M^j_{\ti \vee T_0} - M^j_{T_0}|>\tfrac{1}{2}\eps^\gj}+ \Prob{\eps|V^j_{\ti \vee T_0} - V^j_{T_0}|>\tfrac{1}{2}\eps^\gj}\\
    &\leq 2\exp\big(-\tfrac{1}{2C_\delta}\eps^{2\gj - 2\fl +\delta}\big)+ \Prob{C_\delta \eps^{\fl-\frac{1}{2}\delta+1}>\tfrac{1}{2}\eps^\gj}= \smallo{\eps^{\rho_i}}\text{, for all }j.
\end{split}
\end{align}

To see the upper bound in Lemma \ref{lemma:rough_est}, observe that
\begin{align*}
    &\Prob{Y^j_\tau\in [a^j,b^j],\forallj; Y^i_\tau = L} \leq \Prob{Y^j_\ti \in [a^j, b^j], \forallj;Y^i_\ti =L } + \smallo{\eps^{\rho_i}}\\
    &= \ProbBig{y^j+U^j_\ti \in L^{-\fl}\eps^{\fl - 1}|y^i+U^i_\ti|^\fl[a^j,b^j],\forallj;\ y^i+U^i_\ti>0}+ \smallo{\eps^{\rho_i}}\\
    &\leq \PP \Big\{ y^j+U^j_{\ti\vee T_0} \in L^{-\fl}\eps^{\fl - 1}|y^i+U^i_{\ti\vee T_0}|^\fl[a^j,b^j],\forallj;\\
    &\hspace{7cm} y^i+U^i_{\ti\vee T_0}\in\big(0,(\log\eps^{-1})^{\kappa+1}\big]\Big\}+ \smallo{\eps^{\rho_i}}\\
    &\leq \Prob{y+U_{T_0}\in B_+}+\smallo{\eps^{\rho_i}}
\end{align*}
where we used Lemma \ref{lemma:drop_tau_i} in the first inequality, \eqref{eq:MainProbability} in the identity,\eqref{eq:ti>t_0_whp}  in the second inequality, \eqref{eq:U^j_ti_U^j_t_0_difference_est} in the third inequality.

For the lower bound, we have
\begin{align*}
    &\Prob{Y^j_\tau\in [a^j,b^j],\forallj; Y^i_\tau = L}\\
    &\geq  \ProbBig{y^j+U^j_\ti \in L^{-\fl}\eps^{\fl - 1}|y^i+U^i_\ti|^\fl[a^j,b^j],\forallj;\quad y^i+U^i_\ti>0}- \smallo{\eps^{\rho_i}}\\
    &\geq \PP \Big\{ y^j+U^j_{\ti\vee T_0} \in L^{-\fl}\eps^{\fl - 1}|y^i+U^i_{\ti\vee T_0}|^\fl[a^j,b^j],\forallj;\\
    &\hspace{7cm} y^i+U^i_{\ti\vee T_0}\in\big(0,(\log\eps^{-1})^{\kappa+1}\big]\Big\}- \smallo{\eps^{\rho_i}}\\
    &\geq \PP \Big\{ y^j+U^j_{\ti\vee T_0} \in L^{-\fl}\eps^{\fl - 1}\Big(|y^i+U^i_{ T_0}|-\eps^\gi\Big)^\fl[a^j,b^j],\forallj;\\
    &\hspace{7cm}  y^i+U^i_{ T_0}\in\big(\eps^\gi,(\log\eps^{-1})^{\kappa+1}-\eps^\gi\big]\Big\}- \smallo{\eps^{\rho_i}}\\
    &\geq \Prob{y+U_{T_0}\in B_-}-\smallo{\eps^{\rho_i}}
\end{align*}
where we used Lemma \ref{lemma:drop_tau_i} and \eqref{eq:MainProbability} for the first inequality, \eqref{eq:ti>t_0_whp} for the second inequality, \eqref{eq:U^j_ti_U^j_t_0_difference_est} for the last two inequalities. We remark that in the penultimate inequality the factor $(|y^i+U^i_{ T_0}|-\eps^\gi)^\fl$ is well-defined on the event we consider. This completes our proof of Lemma \ref{lemma:rough_est}.
\end{proof}

In order to prove Lemma \ref{lemma:iteration}, we recall the density estimates obtained in~\cite[Lemma~4.1]{exit_time} for the same setup and assumptions as in the present paper. For a random variable $\mathcal{X}$ with its value in $\R^d$, its density, if exists, is denoted as $\rho_\mathcal{X}$. Since $U_t$ depends on $y$, we denote its density by $\rho^y_{U_t}$. 
\begin{Lemma}\label{Lemma:density_est}
     Consider \eqref{eq:Y_after_Duhamel} with $Y_0=\eps y$. Let  $\pp(x)=\sum_{j,k=1}^dx^\frac{\lj}{\lk}$ for $x\geq 0$ and
    \begin{align}\label{eq:def_Z}
    Z^j_t = \int_0^t e^{-\lj s}F^j_l(0)dW^l_s.
    \end{align}
   Then 
    \begin{enumerate}
        \item there is a constant $\theta>0$ such that for each $\upsilon \in (0,1)$ there are $C,c,\delta>0$ such that, for $\eps$ sufficiently small,
    \begin{align*}
        |\rho^y_{U_{T(\eps)}}(x)-\rho_{Z_{T(\eps)}}(x)|\leq C\eps^\delta\big(1+\pp(\eps^{1-\upsilon}|y|)\big) e^{-c|x|^2},\quad  x, y\in \R^d, 
    \end{align*}
        holds for all deterministic $T(\eps)$ with $1\leq T(\eps)\leq \theta\log\eps^{-1}$;
    \item for each $\theta'>0$, there are constants $C',c',\delta'$ such that, for $\eps$ sufficiently small,
    \begin{align*}
          |\rho_{Z_{T(\eps)}}(x)-\rho_{Z_\infty}(x)|\leq C'\eps^{\delta'} e^{-c'|x|^2},\quad  x\in \R^d, 
    \end{align*}
        holds for all deterministic $T(\eps)$ with $T(\eps)\geq \theta'\log\eps^{-1}$,
    \end{enumerate}
    
\end{Lemma}

We will derive the following result from Lemma~\ref{Lemma:density_est} and use it 
to prove Lemma~\ref{lemma:iteration}.

\begin{Lemma}\label{lemma:terminal-iteration}
For each $\upsilon\in(0,1)$, there is $\delta>0$ such that
\begin{align*}
    \sup_{|y|\leq \eps^{\upsilon-1}}\big|\Prob{y+U_{T_0}\in B_\pm} -  \Prob{y+Z_{T_0}\in B_\pm}\big|=\smallo{\eps^{\rho_i+\delta}}.
\end{align*}    
\end{Lemma}

Let us derive Lemma~\ref{lemma:iteration} from these lemmas first and prove 
Lemma~\ref{lemma:terminal-iteration} after that.

\begin{proof}[Proof of Lemma \ref{lemma:iteration}]

The definition~\eqref{eq:def_Z} implies that $Z_\infty$ is well-defined and has the same distribution as $\Zc$. The definition of $B_\pm$ in \eqref{eq:defOfB_pm} implies that there is $p>0$ such that for small $\eps$,
\begin{align}\label{eq:contain_B_pm}
    B_\pm \subset \prod_{j=1}^d \Big(\eps^{\fl -1}(\log \eps^{-1})^p[-1,1]\Big).
\end{align}

Since there is some $\theta'>0$ such that $T_0\geq \theta'\log\eps^{-1}$, by part (2) of Lemma \ref{Lemma:density_est} \and~\eqref{eq:contain_B_pm}, we obtain that, for $\eps$ sufficiently small, 
\begin{align*}
   |\Prob{y+Z_{T_0}\in B_\pm} -  \Prob{y+\Zc\in B_\pm}\big|=\smallo{\eps^{\rho_i+\delta'} (\log \eps^{-1})^{pd}},\quad\forall y\in\R^d.
\end{align*}
The above display and Lemma \ref{lemma:terminal-iteration} together imply the result of Lemma \ref{lemma:iteration}.
 \end{proof}

To prove 
Lemma~\ref{lemma:terminal-iteration}, we need some notation. 
For $\mathbf{v}\in \R^d$, $A\subset \R^d$ and $t\in \R$, we write $e^{\lam t}\mathbf{v} = (e^{\lj t}\mathbf{v}^j)_{j=1}^d\in\R^d$ and $e^{\lam t}A=\{e^{\lam t}x: x\in A \}\subset \R^d$. 

Recalling $T_0=T_0(\eps)$ from \eqref{eq:def_T_0} and $\theta$ from the statement of Lemma~\ref{Lemma:density_est}, we set $N=\min\{n\in\mathbb{N}:\frac{T_0}{n}\leq \theta\log\eps^{-1}, \forall \eps \in(0,1/2]\}$ and $t_k=\frac{k}{N}T_0$.

Lemma~\ref{lemma:terminal-iteration} is a specific case of the following result
with $k=N$ and $w=0$:

\begin{Lemma}\label{lemma:iterative_est}
For each $\upsilon\in(0,1)$, there is a constant $\upsilon'$ and constants $\eps_k,C_k,\delta_k$, $k=1,2,...,N$ such that
\begin{align}\label{eq:induction_step_conclusion}
    \sup_{|y|\leq \eps^{\upsilon-1}}\sup_{|w|\leq \eps^{\upsilon'-1}}\big|\Probx{\eps y}{y+U_{t_k}+e^{-\lam t_k}w\in B_\pm} -  \Prob{y+Z_{t_k}+e^{-\lam t_k}w\in B_\pm}\big|\leq C_k \eps^{\rho_i+\delta_k},
\end{align}
holds for all $k=1,2,..., N$ and $\eps\in(0,\eps_k]$.
\end{Lemma}

\begin{proof}[Proof of Lemma \ref{lemma:iterative_est}]
First, let us choose $\upsilon'\in(0,1)$ to satisfy
\begin{align}\label{eq:condition_upsilon'}
    \frac{1}{N}\fl \geq \frac{1}{N}\frac{\lambda_d}{\li} >\upsilon', \quad \text{ for all }j=1,2,...,d. 
\end{align}

For the case $k=1$, Lemma \ref{Lemma:density_est} and \eqref{eq:contain_B_pm} imply that
\begin{align*}
     &\sup_{|y|\leq \eps^{\upsilon-1}}\sup_{|w|\leq \eps^{\upsilon'-1}}\big|\Probx{\eps y}{y+U_{t_1}+e^{-\lam t_1}w\in B_\pm} -  \Prob{y+Z_{t_1}+e^{-\lam t_1}w\in B_\pm}\big|\\
     &\leq \sup_{|y|\leq \eps^{\upsilon-1}}\sup_{|w|\leq \eps^{\upsilon'-1}} \int_{\{x\in\R^d:y+x+e^{-\lam t_1}w\in B_\pm\}}C\eps^\delta\big(1+\pp(\eps^{1-\upsilon}|y|)\big)e^{-c|x|^2}dx\leq C \eps^\delta |B_\pm|\leq C_1\eps^{\rho_i+\delta_1}
\end{align*}
for some $C_1, \delta_1>0$.

We proceed by induction. 
Let $k\leq N$ and let us assume that \eqref{eq:induction_step_conclusion} holds for $k-1$.

Set $z(u)= e^{\lam t_{k-1}}(y+u)$. The Markov property of $Y_t$ implies that 
\begin{align*}
    \begin{split}
        &\Probx{\eps y}{y+U_{t_k}+e^{-\lam t_k}w\in B_\pm}  =  \Probx{\eps y}{Y_{t_k}+\eps w\in \eps e^{\lam t_k}B_\pm} \\ 
        &= \EX{\eps y}{\Probx{ Y_{t_{k-1}}}{Y_{t_1}+\eps w\in \eps e^{\lam t_k}B_\pm}} \\
        & = \EXBig{\eps y}{\Probx{ \eps z(u)}{z(u)+U_{t_1}+e^{-\lam t_1}w\in e^{\lam t_{k-1}}B_\pm}\big|_{u=U_{t_{k-1}}}}.
    \end{split}
\end{align*}

To check~\eqref{eq:induction_step_conclusion} for $k$ and complete the induction
step, we must show that
the error caused by replacing $U_{t_1}$ by $Z_{t_1}$ and $U_{t_{k-1}}$ by $Z_{t_{k-1}}$ in this expression is small.
More precisely,  \eqref{eq:induction_step_conclusion} for $k$ will follow immediately
once we prove that there are $\eps_k,\delta',\delta''>0$ such that
the following relations hold uniformly in $|y|\leq\eps^{\upsilon-1}$, $|w|\leq \eps^{\upsilon'-1}$  and $\eps \in(0,\eps_k]$:
\begin{equation}
\label{eq:replacement-1}
|\E^{\eps y}A_\eps(U_{t_{k-1}},w)- \E^{\eps y} B_\eps(U_{t_{k-1}},w)|=o(\eps^{\rho_i+\delta'})
\end{equation}
and
\begin{equation}
\label{eq:replacement-2}
\big|\E^{\eps y}{ B_\eps(U_{t_{k-1}},w)}-C_\eps(y,w)\big|
=o(\eps^{\rho_i+\delta''}),
\end{equation}
where
\begin{align*}
A_\eps(u,w)&=\Probx{ \eps z(u)}{z(u)+U_{t_1}+e^{-\lam t_1}w\in e^{\lam t_{k-1}}B_\pm},
\\
B_\eps(u,w)&=\Prob{z(u)+Z_{t_1}+e^{-\lam t_1}w\in e^{\lam t_{k-1}}B_\pm},
\\
C_\eps(y,w)&=\Prob{y+Z_{t_k}+e^{-\lam t_k}w\in B_\pm}.
\end{align*}

Let us derive \eqref{eq:replacement-1}.
Due to part (1) of Lemma \ref{Lemma:density_est}, there are $\delta', C',c'>0$ such that
\begin{align}\label{eq:|A_eps(u,w)-B_eps(u,w)|_est}
|A_\eps(u,w)-B_\eps(u,w)|\leq \int_{\{x\in\R^d:z(u)+x+e^{-\lam t_1}w\in e^{\lam t_{k-1}}B_\pm\}}C'\eps^{\delta'}\big(1+\pp(\eps^{1-\upsilon'}|z(u)|)\big)e^{-c'|x|^2}dx.
\end{align}
By \eqref{eq:condition_upsilon'}, we have, for $\eps$ sufficiently small,
\begin{align*}
    e^{\lj t_{k-1}}\eps^{\fl -1}(\log\eps^{-1})^p\leq e^{\lj t_{N-1}}\eps^{\fl -1}(\log\eps^{-1})^p \leq  \eps^{\frac{1}{N}\fl -1 }(\log\eps^{-1})^p <  \eps^{\upsilon'-1}.
\end{align*}

This, together with \eqref{eq:contain_B_pm}, implies that there is a constant $C>0$, 
such that\\
if $z(u)+x+e^{-\lam t_1}w\in e^{\lam t_{k-1}}B_\pm$  and $|w|\leq \eps^{\upsilon'-1}$, then
\begin{align}\label{eq:|z(u)|_est}
    \eps^{1-\upsilon'}|z(u)| \leq C+\eps^{1-\upsilon'}|x|.
\end{align}
Using $e^{-c'|x|^2}$ to absorb polynomials of $|x|$, from \eqref{eq:|A_eps(u,w)-B_eps(u,w)|_est} and \eqref{eq:|z(u)|_est} we obtain, for some $C,c>0$,
\begin{align*}
     |A_\eps(u,w)-B_\eps(u,w)|\leq \eps^{\delta'}\int_{\{x\in\R^d:z(u)+x+e^{-\lam t_1}w\in e^{\lam t_{k-1}}B_\pm\}}Ce^{-c|x|^2}dx, \qquad |w|\leq \eps^{\upsilon'-1}.
\end{align*}
Let $\NN$ be a centered Gaussian with density proportional to $e^{-c|x|^2}$ and independent of~$\mathcal{F}_{t_{k-1}}$. The above display implies that if $|w|\leq \eps^{\upsilon'-1}$, then
\begin{align*}
|\E^{\eps y}A_\eps(U_{t_{k-1}},w)- \E^{\eps y} B_\eps(U_{t_{k-1}},w)|
        \leq C\eps^{\delta'}\Probx{\eps y}{y+U_{t_{k-1}}+e^{-\lam t_k}w+e^{-\lam t_{k-1}}\NN \in B_\pm}.
\end{align*}
Each entry of $e^{-\lambda t_1}$ decays like a small positive power of $\eps$. So, for small $\eps$,
\begin{align}\label{eq:induction_assumption_satisfied}
    |w| \leq \eps^{\upsilon'-1} \quad\text{ implies }\quad |e^{-\lam t_1}w|+ \log\eps^{-1}\leq \eps^{\upsilon'-1}.
\end{align}

Therefore,
\begin{align*}\begin{split}
&|\E^{\eps y}A_\eps(U_{t_{k-1}},w)- \E^{\eps y} B_\eps(U_{t_{k-1}},w)|
\\
        &\leq C\eps^{\delta'}\Probx{\eps y}{y+U_{t_{k-1}}+e^{-\lam t_{k-1}}(e^{-\lam t_1}w+\NN) \in B_\pm;\ |\NN|\leq  \log\eps^{-1}} + \smallo{\eps^{\rho_i+\delta'}}\\
    &\leq C\eps^{\delta'}\Prob{y+Z_{t_{k-1}}+e^{-\lam t_{k-1}}(e^{-\lam t_1}w+\NN)\in B_\pm}+\smallo{\eps^{\rho_i+\delta_{k-1}+\delta'}}+\smallo{\eps^{\rho_i+\delta'}}\\
    & = \smallo{\eps^{\rho_i+\delta'}}, 
\end{split}
\end{align*}
uniformly in $|y|\leq \eps^{\upsilon-1}$ and $|w|\leq \eps^{\upsilon'-1}$.
Here, in the second inequality we used the induction assumption allowed by \eqref{eq:induction_assumption_satisfied}, independence of $\NN$, Fubini's theorem, and the superpolynomial decay of $\Prob{|\NN|> \log\eps^{-1}}$. In the last line we used \eqref{eq:contain_B_pm}, the uniform boundedness of the density of $Z_{t_{k-1}}$, independence of $\NN$ and Fubini's theorem. This completes the proof of \eqref{eq:replacement-1}.

Let us now prove \eqref{eq:replacement-2}.  Let $\tZ_{t_1}$ be a copy of $Z_{t_1}$ independent of $\mathcal{F}_{t_{k-1}}$. 
The following holds uniformly in $|y|\leq \eps^{\upsilon-1}$ and $|w|\leq\eps^{\upsilon'-1}$:
\begin{align*} 
&\E^{\eps y} B_\eps(U_{t_{k-1}},w)
\\
    &= \Probx{\eps y}{y+U_{t_{k-1}}+e^{-\lam t_{k-1}}(e^{-\lam t_1}w+\tZ_{t_1})\in B_\pm}\\
    & = \Probx{\eps y}{y+U_{t_{k-1}}+e^{-\lam t_{k-1}}(e^{-\lam t_1}w+\tZ_{t_1})\in B_\pm;\ |\tZ_{t_1}|\leq  \log\eps^{-1}}+\smallo{\eps^{\rho_i+\delta'}}\\
    &= \Prob{y+Z_{t_{k-1}}+e^{-\lam t_{k-1}}(e^{-\lam t_1}w+\tZ_{t_1})\in B_\pm;\ |\tZ_{t_1}|\leq  \log\eps^{-1}}+\smallo{\eps^{\rho_i+\delta_{k-1}}}+\smallo{\eps^{\rho_i+\delta'}}\\
    &= \Prob{y+Z_{t_{k-1}}+e^{-\lam t_{k-1}}\tZ_{t_1}+e^{-\lam t_k}w\in B_\pm}+\smallo{\eps^{\rho_i+\delta_{k-1}\wedge \delta'}},
    \\
    &=C_\eps(y,w) +\smallo{\eps^{\rho_i+\delta_{k-1}\wedge \delta'}},
\end{align*}
where we used the induction assumption in the third identity allowed by \eqref{eq:induction_assumption_satisfied}, independence of $\tZ_{t_1}$ and Fubini's theorem. In the last line, we used the identity in distribution between $Z_{t_{k-1}}+e^{-\lam t_{k-1}}\tZ_{t_1}$ and $Z_{t_k}$.
This proves~\eqref{eq:replacement-2} with $\delta''=\delta_{k-1}\wedge \delta'$
completing the induction step and the entire proof.
\end{proof}

\section{Extension to a general domain}\label{section:extension}

\subsection{Proof of Proposition \ref{Prop:geometry_of_pull_back} }\label{subsection:geometry_pullback}
We use the notation introduced in \eqref{eq:def_M^k}--\eqref{eq:def_of-zeta}.

Let us first prove the inequality $L'_A>0$ and part (1). The assumption $i(A)=i$ together with definitions \eqref{eq:def_M^k} and \eqref{eq:def_of-zeta} implies that for any $L<L(O)$, $\overline{\zeta_L(A)}\cap \Lambda^{i-1}=\emptyset$. Let us fix any $L_0<L(O)$. 
Since the set $\overline{\zeta_{L_0}(A)}\subset \partial \BB_{L_0}$ is compact, we can find $r>0$ such that $\max\{|y^m|: m>i-1\}/|y|\ge r$ for all
$y\in \overline{\zeta_{L_0}(A)}$. Let us choose $t_0>0$ such that for all $t>t_0$ the following holds:
if $j\le i-1\le m$ and $|x^m|\le |x^j|$, then $|e^{\lambda_m t}x^m|/|e^{\lambda_j t}x^j|<r$.  Now if $L$ is small
enough to ensure that $\bar S_{t_0}\BB_L \subset \BB_{L_0}$, then 
for every $j\le i-1$ and every $x\in \FF_L^j$, we are guaranteed that the orbit of $x$ under $\bar S$
intersects~$\partial \BB_{L_0}$ at a point $y$ satisfying $\max\{|y^m|: m>i-1\}/|y|\le \max\{|y^m|: m>i-1\}/|y^j|< r$,  so  $y\notin\overline{\zeta_{L_0}(A)}$ and  thus $x\notin \overline{\zeta_L(A)}$, which completes the proof of $L'_A>0$ and
part~(1). 

To prove part (\ref{item:2_of_prop}), we fix $L<L'_A$ arbitrarily and recall that $\indset(A)=i$. It suffices to define
\begin{align*}
B&=\{x\in \partial\BB_L:\ |x_j|\le L/2 \text{\rm\ for all\ } j>i\},\\
C&=\{x\in \partial\BB_L:\ |x_j|> L/2 \text{\rm\ for some\ } j>i\}=\partial\BB_L\sm B,\\
A_0& = \zeta_L^{-1}(\zeta_L(A)\cap B)=A\cap \zeta_L^{-1}(B),\\
A_1&=\zeta_L^{-1}(C).
\end{align*}
Property (2a) is obvious from the construction. The first of (2b) and (2c) hold due to \eqref{item:equiv-index} of Lemma \ref{lemma:Hyperbolic_Pull_Back} and the construction. The second item of (2b) follows from the construction, and the last one follows from part (1) and the definition of $B$. Lastly, the $N$-regularity of $A_0$ and $A_1$ can be verified through the construction and \eqref{item:equiv-Lambda-regular} of Lemma~\ref{lemma:Hyperbolic_Pull_Back}.

\subsection{Proof of Proposition \ref{Prop:measure-theoretical_prop_of_pull-back}}\label{subsection:error_of_pullback} We
recall the definitions of  $\tau$ and $\tau_L$ given in~\eqref{eq:def_tau} and \eqref{eq:def_tau_L}. The goal is to show that the asymptotics of $\Prob{X_\tau \in A}$ is exactly captured by that of $\Prob{X_{\tau_L} \in f^{-1}\circ\zeta_L(A)}$ for suitable $A\subset \partial \DD$, by which we are able to prove Proposition \ref{Prop:measure-theoretical_prop_of_pull-back}.

To this end, we need to approximate $\zeta_L(A)$. Simply taking a small neighborhood of that set makes it difficult to verify the continuity with respect to the measure $\HH^{i-1}(\,\cdot\,\cap \FF^i_{L,\delta}\cap\Lambda^i )$, which is required in Proposition \ref{Prop:Box_Case}. Hence, the following lemma is needed. We recall the definition of $\dist(\cdot,\cdot)$ in~\eqref{eq:dist-def}.

\begin{Lemma}\label{lemma:existence_of_delta_set} Let $d\ge 2$.
Let $A\subset \partial \DD$ be $N$-regular with $\indset(A)=i$. For $L< L_0$ defined in Proposition~\ref{Prop:Box_Case}, let  $B=\zeta_L(A)$ and assume that $\overline{B}\subset \Int_{L}{\FF^i_L}$. Then,  there are two families of Borel sets $(B_\delta)_{\delta>0}$ and $(B_{-\delta})_{\delta>0}$ with the following properties:
 \begin{enumerate}
     \item $B_{-\delta} \subset B \subset B_\delta\subset \FF^i_{L,\delta_1}$ for some $\delta_1 >0$ and all $\delta>0$; \label{item:one}
     \item $\lim_{\delta \to 0}\HH^{i-1}(B_{\pm \delta}\cap \Lambda^i)=\HH^{i-1}(B\cap \Lambda^i)$
     for all $\delta>0$; \label{item:two} 
     \item $B_{\pm \delta}$ are finite unions of rectangles described in Proposition~\ref{Prop:Box_Case}, whose interiors are pairwise disjoint;\label{item:three}
     \item $\mathrm{dist}\big(\partial\BB_L \sm B_\delta, B\big)>0$  and $\mathrm{dist}\big(\partial \BB_L \sm B, B_{-\delta}\big)>0$ for all $\delta>0$.\label{item:four}
 \end{enumerate}
\end{Lemma}

The next result shows that  $\Prob{f(X_{\tau_L})\in   (\zeta_L(A))_{\pm\delta}}$ is 
a very good approximation for~$\Prob{X_\tau \in A}$.

\begin{Lemma} \label{lemma:Est_Pull_Back}
Let $L<L_0$. For each $A\subset \partial \DD$ as in Lemma \ref{lemma:existence_of_delta_set}, there is $\delta_0>0$ depending on $A$ such that for each $\delta\in(0,\delta_0)$
\begin{align}\label{eq:in_lemma:Est_Pull_Back}
\mathcal{O}(e^{-C_\delta\eps^{-2}}) + \Prob{f(X_{\tau_L}) \in (\zeta_L(A))_{-\delta}}
\leq \Prob{X_\tau \in A}
\leq 
 \Prob{f(X_{\tau_L}) \in (\zeta_L(A))_\delta}  +  \mathcal{O}(e^{-C_\delta\eps^{-2}}),
\end{align}
as $\eps \to 0$, for some $C_\delta>0$ depending on $\delta$.
\end{Lemma}

To prove this lemma, we will need an estimate on the discrepancy between the deterministic path $S_tx$ and the perturbed one, i.e., the process $X_t$ under $\mathbb{P}^x=\Prob{\,\cdot\,|X_0=x}$.  To that end, we will use the following consequence of Proposition~2.3 and Theorem~2.4  in \cite[Chapter~III]{Azencott:MR590626} which is an extension of the standard FW large deviation bound without an assumption of uniform ellipticity of $\sigma$. We state it here because we can use it directly for the case $d=1$ in the proof of Proposition \ref{Prop:measure-theoretical_prop_of_pull-back}.

\begin{Lemma} \label{lemma:FW_exponentialDecay}
Let $b$ and $\sigma$ be Lipschitz and bounded. For all $\eps>0$, let $(X_t^\eps)_{t\ge0}$, be a solution of the It\^o equation \eqref{eq:SDE_X} with initial condition $X^\eps_0=x$, under a probability measure $\PP$ and recall the definition of the flow  $(S_t)$ from~\eqref{eq:ODE}.
For each deterministic $T>0$ and $\eta>0$, $$\ProbBig{\sup_{0\leq t \leq T}|X^x_t - S_tx|> \eta}=\mathcal{O}(e^{-C\eps^{-2}})$$ holds uniformly in $x$, where $C$ depends only on $T$, $\eta$ and the Lipschitz constant of $b$.
\end{Lemma}

\begin{proof}[Proof of Proposition \ref{Prop:measure-theoretical_prop_of_pull-back}] First, we consider $d\geq 2$.
Splitting $A$ into two sets if necessary, we can assume that $\overline{\zeta_L(A)}\subset \Int_{L}\FF^i_{L+}$ without loss of generality. 
Let $\delta_1$ be defined by part~\eqref{item:one} of Lemma \ref{lemma:existence_of_delta_set}.
By compactness of $\overline{\zeta_L(A)}$, there is $\Delta\in(0,L_0\wedge \delta_1)$ such that $\zeta_L(A)\subset \FF^i_{L+, \Delta}$.  We use (\ref{item:three}) of Lemma \ref{lemma:existence_of_delta_set} to represent $(\zeta_L(A))_{\pm\delta}$ as a finite union of rectangles with disjoint interiors. Applying Proposition~\ref{Prop:Box_Case}
to these rectangles and noting that the contribution from (perhaps overlapping) boundaries
of these rectangles is~$0$, we obtain
\begin{align*}
    \lim_{\eps\to 0}\eps^{-\rho_i}\Prob{f(X_{\tau_L}) \in (\zeta_L(A))_{\pm\delta}}=L^{-\sum_{j<i}\fl}\EE\chi^i_+(\xi_0)  \HH^{i-1}\{(\zeta_L(A))_{\pm\delta} \cap \Lambda^i \}.
\end{align*}
Therefore, due to (\ref{item:two}) of Lemma~\ref{lemma:existence_of_delta_set} ,
\begin{align*}
    \lim_{\delta\to0}\lim_{\eps\to 0}\eps^{-\rho_i}\Prob{f(X_{\tau_L}) \in (\zeta_L(A))_{\pm\delta}}=L^{-\sum_{j<i}\fl}\EE\chi^i_+(\xi_0)  \HH^{i-1}\{\zeta_L(A) \cap \Lambda^i \}.
\end{align*}
Applying Lemma~\ref{lemma:Est_Pull_Back} we complete the proof for $d\ge 2$.

In the special case $d=1$, we have  $\partial \DD=\{q_-,q_+\}$ with $q_-<0<q_+$, and $i=1$.  It suffices to study $\Prob{X_\tau = q_\pm}$. Note that $f^{-1}(\FF^1_{L\pm,\delta})=\{p_\pm\}$ where $p_\pm = f^{-1}(\pm L)$ satisfy $q_-<p_-<0<p_+<q_+$. Proposition \ref{Prop:Box_Case} implies that
\begin{align*}
    \lim_{\eps\to 0}\Prob{X_{\tau_L}=p_\pm} = \EE \chi_\pm^1(\xi_0).
\end{align*}
Using Lemma~\ref{lemma:FW_exponentialDecay} we conclude that
$ \Prob{X_{\tau_L}=p_\pm;\ X_\tau \neq q_\pm}=\Oc\big(e^{-C\eps^{-2}}\big),$
which immediately implies that
$ \lim_{\eps\to 0}\Prob{X_\tau=q_\pm} = \EE \chi_\pm^1(\xi_0)$
and completes the proof.
\end{proof}

\begin{proof}[Proof of Lemma \ref{lemma:existence_of_delta_set}]

Without loss of generality we may assume that $\overline{B}\subset \Int_{L}{\FF^i_{L+}}$. Let us choose $\delta_1>0$ such that
\begin{align*}\overline{B}\subset \Int_{L}{\FF}^i_{L+,\delta_1}.
\end{align*}

If $i=1$, then ${\FF}^i_{L+,\delta_1}\cap \Lambda^1= B\cap \Lambda^1=\{p\}$, where $p=(L,0,\dots,0)$. Part~\eqref{item:equiv-Lambda-regular} of 
Lemma~\ref{lemma:Hyperbolic_Pull_Back} implies $\HH^0(\partial_L B \cap \Lambda^1)=0$, so $p\not\in \partial_L B $ and thus $p\in \Int_L B$. Hence, we can pick a closed rectangle $R$ on ${\FF}^i_{L+,\delta_1}$ such that $p\in  R \subset \Int_L B$. 
Setting $B_{-\delta}= R$ for all $\delta>0$, we ensure properties  (1)---(4) for $B_{-\delta}$. 
Choosing $\delta_1$ sufficiently small and setting $B_\delta= \FF^i_{L,\delta_1}$, for all $\delta>0$, we ensure properties  (1)---(4) for $B_{\delta}$.

If $i=d$, then ${\FF}^i_{L+,\delta_1}\cap\Lambda^d = {\FF}^i_{L+,\delta_1}$. Since ${\FF}^i_{L+,\delta_1}$ is $(d-1)$-dimensional and flat, the measure $\HH^{d-1}(\,\cdot\, \cap {\FF}^i_{L+,\delta_1}\cap \Lambda^d)=\HH^{d-1}(\,\cdot\, \cap {\FF}^i_{L+,\delta_1})$ can be viewed as the $(d-1)$-dimensional Lebesgue measure restricted on ${\FF}^i_{L+,\delta_1}$. By the standard approximation arguments, we can choose $B_{-\delta}$ and $B_\delta$ to be two unions of finitely many rectangles, which satisfy (1) and (2). Slightly adjusting the rectangles, we can ensure (4). 

For $1<i<d$, we need an extended version of this construction.
We construct the family $(B_\delta)_{\delta>0}$ first. 
Let us define a closed $(i-1)$-dimensional rectangle
\begin{align*}Q=\FF^i_{L+,\delta_1}\cap \Lambda^i=\{x\in\R^d:\ |x_1|,\ldots,|x_{i-1}|\le L-\delta_1;\  x_i=L;\ x_{i+1}=\ldots=x_d=0\}.
\end{align*}
For every $\delta>0$, using the compactness of $\overline B \cap Q$ and the fact that 
\begin{equation}
\label{eq:regularity-for-B}
\HH^{i-1}(\partial_{L}B\cap Q) =0
\end{equation}
(which follows from the regularity of $A$ and Lemma \ref{lemma:Hyperbolic_Pull_Back}), we can
find a set $G_\delta$  satisfying the following:
\begin{gather}
	\text{$G_\delta$ is a finite union of open $(i-1)$-dimensional rectangles};\label{eq:G_delta_property_rect}\\    
    \overline{B}\cap Q\subset G_\delta\subset Q; \label{eq:G_delta_property_containment}\\
    \HH^{i-1}(G_\delta \sm(B\cap Q)) = \HH^{i-1}(G_\delta \sm(\overline{B}\cap Q)) < \delta.\label{eq:G_delta_property_approximation}
\end{gather}
Since $Q\sm G_\delta$ and $\overline{B}$ are compact, we can adjust $G_\delta$ to additionally ensure that 
\begin{align}\label{eq:QG_delta_dist_to_B}
    \mathrm{dist}(Q\sm G_\delta, \overline{B})>0.
\end{align}
Let $\pi$ be the orthogonal projection onto $\Lambda^i$, namely
\begin{align*}
    \pi: x\in\R^d \mapsto (x^1,x^2,...,x^{i-1},0,\dots,0)\in\R^d.
\end{align*}
Since $\FF^i_{L+,\delta_1}\sm\overline{B}$ is open and $Q\sm G_\delta$ is closed in the relative topology of $\FF^i_{L,\delta_1}$, \eqref{eq:QG_delta_dist_to_B} implies that there is some ``thickness'' $h(\delta)\in(0,\delta)$ such that 
\begin{align}\label{eq:K_delta}
    K_\delta = \{x\in \FF^i_{L,\delta_1}: \pi(x)\in Q\sm G_\delta; \quad |x^j|< h(\delta),\ \forall j> i \}
\end{align}
satisfies
\begin{equation}
      \mathrm{dist}(K_\delta, B)>0.\label{eq:dist_K_B}
\end{equation}

Let us define $B_\delta=\FF^i_{L+,\delta_1} \sm K_\delta.$ Parts~(\ref{item:one}) and~(\ref{item:four}) 
of the lemma
now follow from~\eqref{eq:dist_K_B}.

Using~\eqref{eq:G_delta_property_rect} and subdividing rectangles if needed we can represent $\overline{G_\delta}$ as a finite union of $(i-1)$-dimensional closed rectangles with disjoint interiors. 
Part~\eqref{item:three} follows now from~\eqref{eq:K_delta} and the definition
of $B_\delta$.

Since
\begin{equation*}
B_\delta \cap Q=B_\delta \cap \Lambda^i= G_\delta,
\end{equation*}
we have $\HH^{i-1}(B_\delta \cap \Lambda^i)= \HH^{i-1}(G_\delta)$. Thus,
\[
0\le \HH^{i-1}(B_\delta \cap \Lambda^i)-\HH^{i-1}(B \cap \Lambda^i)
=\HH^{i-1}(G_\delta\sm (B\cap Q))<\delta,
\]
by \eqref{eq:G_delta_property_containment} and~\eqref{eq:G_delta_property_approximation}, so part~(\ref{item:two}) also follows.

To construct $B_{-\delta}$, we apply the same approach to the set $B_-=\FF_{L+,\delta_1}\sm B$
and note that due to the regularity of $A$, the set $B_-$ satisfies a version of \eqref{eq:regularity-for-B},
namely, 
\begin{equation*}
\HH^{i-1}(\partial_{L}B_-\cap Q) =0,
\end{equation*}
so we can find a cover $G_{-\delta}$ of $\overline{B_-}\cap Q$ satisfying the versions of requirements
\eqref{eq:G_delta_property_rect}--\eqref{eq:QG_delta_dist_to_B} with $B,G_\delta$ replaced by $B_-,G_{-\delta}$. We can now define $K_{-\delta}$ via $B_-$ and $G_{-\delta}$ similarly to~\eqref{eq:K_delta}--\eqref{eq:dist_K_B}, and check that properties
(1)--(4) hold if we set $B_{-\delta}=K_{-\delta}$.
\end{proof}

\begin{proof}[Proof of Lemma \ref{lemma:Est_Pull_Back}]
To derive the upper bound in~\eqref{eq:in_lemma:Est_Pull_Back}, we write
\begin{align*}
\Prob{X_\tau \in A}\le  \Prob{f(X_{\tau_L}) \in (\zeta_L(A))_\delta} + \mathtt{I},
\end{align*}
where  $\mathtt{I}=\Prob{f(X_{\tau_L}) \notin (\zeta_L(A))_\delta,\ X_\tau \in A}$ is the term we need to estimate.

Let $\Gamma =  f^{-1}(\partial \BB_{L})$.
Since $b$ is transversal to both~$\Gamma$  and $\partial \DD$,
 the inverse of the map~$\psi_L$ defined in~\eqref{eq:psi_L} is Lipschitz on~$\partial\DD$.

Let us introduce $F_\delta=f^{-1}((\zeta_L(A))_\delta)$ and notice that
$\gamma=\dist(\partial \DD\sm\psi_L(F_\delta), A)>0$ due to the Lipschitz property of $\psi_L^{-1}$ and (4) of Lemma \ref{lemma:existence_of_delta_set}.

Let $T_0=\sup\{t(x):\ x\in\Gamma\}$, where $t(\cdot)$ was
defined in~\eqref{eq:deterministic-exit-time}, and $T_1=T_0+1$.  Due to the same transversality properties, by time $T_1$, all orbits under $S$ originating from $\Gamma$ exit $\DD$ and end up at distance from $\partial \DD$ that is bounded away from~$0$. Therefore there is $\eta>0$ such that for every $x\in\Gamma$, and every continuous path
$y:[0,T_1]\to\R^d$  such that $\sup_{t\in[0,T_1]}|y(t)-S_tx|\le \eta$, 
the point $y_\DD$ of the first intersection of the path $y$ with $\partial\DD$ is well-defined and satisfies $|y_\DD-\psi_L(x)|<\gamma$.

We can now apply this statement along with Lemma~\ref{lemma:FW_exponentialDecay} to see that
\begin{align*}
\mathtt{I}&= \int_{\Gamma\sm F_\delta}\Prob{X_{\tau_L}\in dx}\Pro^x\{X_\tau \in A\}
 \\ &\le \int_{\Gamma\sm F_\delta}\Prob{X_{\tau_L}\in dx}\Pro^x\left\{\sup_{0\le t \le T_1} |X^x_t-S_tx| > \eta \right\}=\Oc(e^{-C\eps^{-2}})
\end{align*}
for some $C=C(\delta)>0$, which completes the proof of the upper bound in~\eqref{eq:in_lemma:Est_Pull_Back}.
The lower bound in~\eqref{eq:in_lemma:Est_Pull_Back} is derived similarly.
\end{proof}

\bibliographystyle{alpha}
\end{document}